\documentclass[10pt]{article}

\usepackage{amsthm,amsfonts,amsmath,amssymb}
\usepackage{mathtools}

\usepackage{pxfonts}
\usepackage{euscript}
\usepackage{bbold,bbm}
\usepackage[utf8]{inputenc}

\usepackage{authblk}

\usepackage{hyperref}
\usepackage{cleveref}
\usepackage{graphics}
\usepackage{epstopdf} 
\usepackage{xypic}
\usepackage[all,2cell,arrow,matrix,tips]{xy} \UseAllTwocells \SilentMatrices
\usepackage{graphicx}
\usepackage{verbatim}
\usepackage{leftidx}
\usepackage{enumerate}
\usepackage{phonetic}\usepackage{sectsty}

\usepackage{subcaption} \usepackage{cleveref}

\sectionfont{\Large}

\usepackage[dvipsnames,usenames]{color}
  \usepackage{titleps}
  \newpagestyle{main}[\small]{
    \setheadrule{.55pt}%
    \sethead[\thepage]
            [\chaptertitle]
            []
            {}
            {\sectiontitle}
            {\thepage}
  }
\newcommand\void[1]       {}

\theoremstyle{definition}

\newtheorem{thm}{Theorem}[section]
\newtheorem{prop}[thm]{Proposition}
\newtheorem{cor}[thm]{Corollary}
\newtheorem{lem}[thm]{Lemma}

\theoremstyle{definition}
\newtheorem{defn}[thm]{Definition}
\newtheorem{expl}[thm]{Example}

\newtheorem{rem}[thm]{Remark}
\newtheorem{notation}[thm]{Notation}

\numberwithin{equation}{section}
\numberwithin{thm}{section}

\newcommand\be            {\begin{equation}}
\newcommand\ee            {\end{equation}}
\newcommand\bea           {\begin{eqnarray}}
\newcommand\eea         {\end{eqnarray}}
\newcommand\bnu          {\begin{enumerate}}
\newcommand\enu          {\end{enumerate}}

\allowdisplaybreaks[1]

\newlength{\fighskip} \fighskip=2pt
\newlength{\figvskip} \figvskip=3pt

\newcommand{\pf}{\begin{proof}}
\newcommand{\epf}{\end{proof}}

\newcommand\Rb            {\mathbb{R}}

\newcommand\Hb            {\mathbb{H}}
\newcommand\mO           {\mathcal{O}}

\newcommand\A           {\EuScript{A}}
\newcommand\B           {\EuScript{B}}
\newcommand\C           {\EuScript{C}}
\newcommand\D           {\EuScript{D}}
\newcommand\E          {\EuScript{E}}

\newcommand\CK         {\EuScript{K}}
\newcommand\CL          {\EuScript{L}}
\newcommand\M          {\EuScript{M}}
\newcommand\N         {\EuScript{N}}

\newcommand\CP         {\EuScript{P}}
\newcommand\CQ         {\EuScript{Q}}

\newcommand\V        {\EuScript{V}}
\newcommand\W        {\EuScript{W}}

\newcommand{\FZ}{\mathfrak{Z}}

\newcommand\id            {\mathrm{id}}
\newcommand\op          {\mathrm{op}}

\newcommand\vect    {\EuScript{V}\mathrm{ec}}

\newcommand\Fun     {\mathrm{Fun}}
\newcommand\LMod  {\mathrm{LMod}}
\newcommand\RMod  {\mathrm{RMod}}
\newcommand\BMod {\mathrm{BMod}}

\newcommand\forget  {\mathbf{f}}

\newcommand\bk       {\mathbb{k}}

\newcommand\Alg    {\mathrm{Alg}}

\newcommand\unit    {\mathbb{1}}
\newcommand{\rev} {\mathrm{rev}}

\newcommand{\homm} {\mathrm{Hom}}

\newcommand{\cat}{\mathrm{Cat}_{\EuScript{E}}^{\mathrm{fs}}} 
\newcommand{\fcat}{\mathrm{Cat}^{\mathrm{fs}}}

\newcommand{\bal}{\mathrm{bal}}

\newcommand{\boe}{\boxtimes_{\E}}

\newcommand\bit {\begin{itemize}}
\newcommand\eit {\end{itemize}}

\newcommand{\MFus}{\EuScript{MF}\mathrm{us}_{/\EuScript{E}}}
\newcommand{\BFus}{\EuScript{BF}\mathrm{us}_{/\EuScript{E}}}
\newcommand{\BFCe}{\mathrm{BFC}_{/\EuScript{E}}}
\newcommand{\NBFCe}{\mathrm{NBFC}_{/\EuScript{E}}}

\usepackage{tikz}
\usepackage{tikz-cd} 
\tikzset{shorten <>/.style={shorten >=#1,shorten <=#1}}

\title{Integrate on a closed stratified surface}

\author[a,b]{Xiao-Xue Wei \thanks{Email: \href{mailto:xxwei@mail.ustc.edu.cn}{\tt xxwei@mail.ustc.edu.cn}}}

\affil[a]{School of Mathematical Sciences, \authorcr University of Science and Technology of China, Hefei, 230026, China} 
\affil[b]{Shenzhen Institute for Quantum Science and Engineering, \authorcr Southern University of Science and Technology, Shenzhen, 518055, China}

\begin{document}
\date{}
\maketitle
\begin{abstract}
We prove that the Drinfeld center centralized by a symmetric fusion category is a symmetric monoidal functor if we choose proper domain and codomain categories.  
We also compute the factorization homology of stratified surfaces with coefficients satisfying certain anomaly-free conditions.
\newline 
$\textbf{Key words}:$ algebras; centers; stratified surfaces; factorization homology.
\newline
$\textbf{2020 Mathematics Subject Classification}:$  18A30; 18D70; 18M05; 18N20.
\end{abstract}

\tableofcontents

\section{Introduction}
The mathematical theory of factorization homology has been established by Beilinson, Drinfeld\cite{BD, FG}, Lurie \cite{Lu}, Francis \cite{F} and many other people (see for example \cite{CG,AF, AFT1, AFT2, AFR, BBJ, BBJ2}).
Factorization homology takes two input data: an n-manifold $\M$ and an n-disk algebra $\A$ in a symmetric monoidal $\infty$-category $\V$.
The output data $\int_{\M} \A$ is an object in $\V$.
For example, integrate on a compact 1-maifold $S^1$ with an associative algebra $A$. The result $\int_{S^1} A$ is the usual Hochschild homology $HH_{\bullet}(A)$.
 On a compact 2-manifold, the computation is highly nontrivial (see for example \cite{BBJ, BBJ2, Francis}). 
 Motivated by the study of topological orders in condensed matter physics, Ai, Kong and Zheng \cite{LiangFH} integrate a unitary modular tensor category (UMTC) $\A$ on a closed stratified surface $\Sigma$. They get
 $\int_\Sigma \A \simeq (\Hb, u_{\Sigma})$.
Here $\Hb$ is the category of finite-dimensional Hilbert spaces and $u_{\Sigma}$ is a distinguished object in $\Hb$.
In physics, the category $\A$ is the category of anyons in a 2d anomaly-free topological order. 
The result $\int_\Sigma \A$ is a global observable defined on $\Sigma$. 
And this global observable is precisely the ground state degeneracy (GSD) of the 2d topological order on $\Sigma$. This fact remains to be true even if we introduce anomaly-free defects of codimension 1 and 2.
  
A finite onsite symmetry of a 2d symmetry enriched topological (SET) order can be mathematically described by a symmetric fusion category $\E$, and the category of anyons in this SET order can be described by a UMTC over $\E$, which is a unitary braided fusion category with M\"{u}ger center given by $\E$ \cite{TL}. This motivates us to compute the factorization homology on a 2-manifold with the coefficient given by a UMTC over $\E$. 
In \cite{Wei}, Wei integrates a UMTC over $\E$ on a closed stratified surface $\Sigma_g$ with genus $g$. 
The result is a pair $(\E, u_{\Sigma_g})$, where $u_{\Sigma_g}$ is a distinguished object in $\E$.
In this work, we extend the result.
We integrate on any closed stratified surface with anomaly-free coefficients valued in the symmetric monoidal 2-category $\cat$ of finite semisimple module categories over $\E$.

\medskip
The layout of this paper is as follows. In Sec.\,2, we introduce braided $\E$-modules, the tensor product between braided $\E$-modules, and the relation between this tensor product and modular extensions.
In Sec.\,3, we show that the Drinfeld center centralized by $\E$ can be made functorial if we choose a proper domain and codomain categories (see Thm.\,\ref{thm-func-E}).
In Sec.\,4, we compute the factorization homology on any closed stratified surface with anomaly-free coefficients in $\cat$ (see Thm.\,\ref{main-result}).

\bigskip
\noindent $\textbf{Acknowledgement}$ I thank Liang Kong for introducing me to this interesting subject. I am supported by NSFC under Grant No. 11971219 and Guangdong Provincial Key Laboratory (Grant No.2019B121203002).

\section{$\E$-module braidings and braided $\E$-modules}
\subsection{Braided $\E$-modules}
\begin{notation}
Let $\bk$ be an algebraically closed field of characteristic zero.
Let $\E$ be a symmetric fusion category with the braiding $r$.
The category $\vect$ denotes the category of finite dimensional vector spaces over $\bk$ and $\bk$-linear maps.
\end{notation}

Let $\A$ be a monoidal category.
$\A^{\op}$ denotes the monoidal category $\A$,  but the morphism space is $\homm_{\A^{\op}}(a, b) \coloneqq \homm_{\A}(b, a)$ for any objects $a, b \in \A$.
$\A^{\rev}$ denotes the category $\A$ with the reversed tensor product $a \otimes^{\rev} b \coloneqq b \otimes a$ for $a, b \in \A$. 
 A monoidal category $\A$ is rigid if every object $a \in \A$ has a left dual $a^L$ and a right dual $a^R$.
For a braided monoidal category $\A$ with a braiding $c_{a,b}: a \otimes b \rightarrow b \otimes a$,  
 $\overline{\A}$ denotes the monoidal category $\A$ with the anti-braiding $\bar{c}_{a,b} = c^{-1}_{b,a}$. 

A \emph{finite semisimple left $\E$-module} is a finite semisimple $\bk$-linear abelian category $\M$ equipped with a $\bk$-bilinear functor $\odot: \E \times \M \rightarrow \M$, a natural isomorphism $\lambda_{e, e', m}: (e \otimes e') \odot m \simeq e \odot (e' \odot m)$, and a unit isomorphism $l_m: \unit_{\E} \odot m \simeq m$ for all $e, e' \in \E, m \in \M$ and the tensor unit $\unit_{\E} \in \E$ satisfying some natural conditions.
For finite semisimple left $\E$-modules $\M$ and $\N$, a left $\E$-module functor from $\M$ to $\N$ is a pair $(F, s^F)$, where $F: \M \rightarrow \N$ is a $\bk$-linear functor and  $s^F_{e,m}: F(e \odot m) \simeq e \odot F(m)$, $e \in \E$, $m \in \M$, is a natural isomorphism, satisfying some natural conditions.
$\Fun_{\E}(\M, \N)$ denotes the category of left $\E$-module functors $\M \to \N$ and left $\E$-module natural transformations. 
If $\E = \vect$, $\Fun_{\E}(\M, \N)$ is abbreviated to $\Fun(\M, \N)$.
\begin{defn}
\label{defn-2}
Let $\C$ be a finite semisimple left $\E$-module. An \emph{$\E$-module braiding} on $\C$ is a natural isomorphism 
\[ \tau_{e, x}: e \odot x \xrightarrow{\tau^1_{e,x}} x \odot e \xrightarrow{\tau^2_{x, e}} e \odot x \]
 for $e \in \E, x \in \C$ such that $\tau_{\unit_{\E}, x} = \id_{\unit_{\E} \odot x}$ and the diagrams
 \begin{equation} 
 \label{diag-emb-1}
 \begin{split}
 \xymatrixcolsep{0.5pc}
 \vcenter{\xymatrix@=3ex{
 e \odot e' \odot x \ar[rr]^{\tau^1_{e, e' \odot x}} \ar[rd]_{r_{e, e'}, 1} & & e' \odot x \odot e \\
 & e' \odot e \odot x \ar[ru]_{1, \tau^1_{e, x}} &
 }}
 \qquad
 \vcenter{\xymatrix@=3ex{
 e' \odot x \odot e \ar[rd]_{1, \tau^2_{x, e}} \ar[rr]^{\tau^2_{e' \odot x, e}} & & e \odot e' \odot x \\
 & e' \odot e \odot x \ar[ru]_{r_{e',e}, 1} &
 }} 
 \end{split}
 \end{equation}
\begin{equation}
\label{diag-emb-2}
\begin{split}
\xymatrixcolsep{0.5pc}
\vcenter{\xymatrix@=3ex{
e \odot e' \odot x \ar[rr]^{\tau^1_{e \otimes e', x}} \ar[rd]_{1, \tau^1_{e', x}} & & x \odot e \odot e' \\
& e \odot x \odot e'  \ar[ru]_{\tau^1_{e, x}, 1} &
}}
\qquad
\vcenter{\xymatrix@=3ex{
x \odot e \odot e' \ar[rr]^{\tau^2_{x, e \otimes e'}} \ar[rd]_{\tau^2_{x,e}, 1} & & e \odot e' \odot x \\
& e \odot x \odot e' \ar[ru]_{1, \tau^2_{x, e'}} &
}}
 \end{split}
 \end{equation}
 commute for all $e, e' \in \E, x \in \C$. 
 A \emph{braided $\E$-module} is a finite semisimple left $\E$-module with an $\E$-module braiding.
 An $\E$-module braiding is trivial, if $\tau_{e, x} = \id_{e \odot x}$ for all $e \in \E, x \in \C$.
 
 Given two braided $\E$-modules $(\C, \tau^{\C})$ and $(\D, \tau^{\D})$,   
 a \emph{braided $\E$-module functor} is a left $\E$-module functor $(F, s^F): \C \to \D$ such that the diagram
\begin{equation}
\label{braided-E-functor}
\begin{split}
 \xymatrix{
F(e \odot x) \ar[r]^{s^F_{e, x}} \ar[d]_{\tau^{\C}_{e,x}} & e \odot F(x)\ar[d]^{\tau^{\D}_{e, F(x)}} \\
F(e \odot x) \ar[r]_{s^F_{e, x}} & e \odot F(x)
}
\end{split}
\end{equation}
 commutes for $e \in \E, x \in \C$.
We use $\Fun'_{\E}(\C, \D)$ to denote the full subcategory of $\Fun_{\E}(\C, \D)$ formed by braided $\E$-module functors, and left $\E$-module natural transformations.
\end{defn}

\begin{prop}
The following diagrams commute for $e, e' \in \E, x \in \C$:
\begin{equation}
\label{diag-eb-4}
\begin{split}
\xymatrixcolsep{0.5pc}
\xymatrixrowsep{0.5pc}
\vcenter{ \xymatrix@=3ex{
e \odot x \odot e' \ar[rr]^{\tau^1_{e, x \odot e'}} \ar[dr]_{\tau^1_{e, x}, 1} & & x \odot e' \odot e \\
 & x \odot e \odot e' \ar[ru]_{1, r_{e, e'}} &
} }
\qquad 
\vcenter{\xymatrix@=3ex{
x \odot e' \odot e \ar[rr]^{\tau^2_{x \odot e', e}} \ar[rd]_{1, r_{e', e}} & & e \odot x \odot e' \\
& x \odot e \odot e'  \ar[ru]_{\tau^2_{x, e}, 1} &
}}
\end{split}
\end{equation}
\end{prop}

\begin{proof}
Consider the following diagram
\begin{equation}
\label{eq-3}
\begin{split}
 \xymatrix@=4ex{
e \odot e' \odot x \ar[r]^{r_{e, e'}, 1} \ar[d]_{1, \tau^1_{e', x}} \ar@/^1.8pc/[rr]^{\tau^1_{e, e' \odot x}} \ar[rd]|{\tau^1_{e \otimes e', x}} & e' \odot e \odot x \ar[r]^{1, \tau^1_{e, x}} \ar[dr]|{\tau^1_{e' \otimes e, x}} & e' \odot x \odot e \ar[d]^{\tau^1_{e', x}, 1} \\
e \odot x \odot e' \ar[r]_{\tau^1_{e, x}, 1} & x \odot e \odot e' \ar[r]_{1, r_{e, e'}} & x \odot e' \odot e 
} 
\end{split}
\end{equation}
The left triangle, right triangle and upper triangle commute by the diagrams (\ref{diag-emb-1}) and (\ref{diag-emb-2}). The middle square commutes by the naturality of $\tau^1_{e, x}$, $e \in \E, x \in \C$.

\[ \xymatrix{
e \odot e' \odot x \ar[rr]^{\tau^1_{e, e' \odot x}} \ar[d]_{1, \tau^1_{e', x}} & &  e' \odot x \odot e \ar[d]^{\tau^1_{e', x}, 1} \\
e \odot x \odot e' \ar@/^1.5pc/[rr]|{\tau^1_{e, x \odot e'}} \ar[r]_{\tau^1_{e, x}, 1} & x \odot e \odot e' \ar[r]_{1, r_{e, e'}} & x \odot e' \odot e 
} \]
By the commutativiry of the diagram (\ref{eq-3}) and the naturality of $\tau^1$, the bottom triangle commutes.
One can check that the right diagram of (\ref{diag-eb-4}) commutes.
\end{proof}

\begin{prop}
The following diagrams commute for $e, e' \in \E, x \in \C$:
\begin{equation}
\label{diag-embxee'}
\begin{split}
\xymatrixcolsep{0.4pc} 
\vcenter{\xymatrix@=3ex{
e \odot e' \odot x \ar[rr]^{\tau^1_{e \otimes e', x}} \ar[rd]_{\tau^1_{e, e' \odot x}} & & x \odot e \odot e' \\
& e' \odot x \odot e  \ar[ru]_{\tau^1_{e', x \odot e}} &
}}
\qquad
\vcenter{\xymatrix@=3ex{
x \odot e \odot e' \ar[rr]^{\tau^2_{x, e \otimes e'}} \ar[rd]_{\tau^2_{x \odot e, e'}} & & e \odot e' \odot x \\
& e' \odot x \odot e \ar[ru]_{\tau^2_{e' \odot x, e}} &
}} 
\end{split}
\end{equation}
\end{prop}

\begin{proof}
Consider the following diagram:
\[ \xymatrix@=4ex{
e \odot e' \odot x \ar[r]^{\tau^1_{e, e' \odot x}} \ar[d]_{r_{e, e'}, 1} \ar@/^1.7pc/[rr]^{\tau^1_{e \otimes e', x}} & e' \odot x \odot e \ar[r]^{\tau^1_{e', x \odot e}} \ar[rd]_(0.4){\tau^1_{e', x}, 1} & x \odot e \odot e' \ar[d]^{1, r_{e, e'}} \\
e' \odot e \odot x \ar[rr]_{\tau^1_{e' \otimes e, x}} \ar[ru]_(0.6){1, \tau^1_{e, x}} & & x \odot e' \odot e 
} \]
The outward square commutes by the naturality of $\tau^1$. The middle three triangles commute by the diagrams (\ref{diag-emb-1}), (\ref{diag-emb-2}) and (\ref{diag-eb-4}). Then the upper triangle commutes.
Check that the right diagram of (\ref{diag-embxee'}) commutes.
\end{proof}

\begin{prop}
The following diagram commutes for $e, e' \in \E, x \in \C$:
\begin{equation}
\begin{split} 
\label{diag-SL}
\vcenter{\xymatrix{
e \odot e' \odot x \ar[r]^{r_{e, e'}, 1} \ar[d]_{\tau_{e, e' \odot x}} & e' \odot e \odot x \ar[d]^{1, \tau_{e, x}} \\
e \odot e' \odot x \ar[r]_{r_{e, e'}, 1} & e' \odot e \odot x 
}}
\qquad 
 \vcenter{\xymatrix{
e \odot e' \odot x \ar[r]^{1, \tau_{e', x}} \ar[rd]_{\tau_{e \otimes e', x}} & e \odot e' \odot x \ar[d]^{\tau_{e, e' \odot x}}\\
& e \odot e' \odot x 
}} 
\end{split}
\end{equation}
\end{prop}

\begin{proof}
Consider the diagram:
\[ \xymatrix{
e \odot e' \odot x \ar[r]^{r_{e, e'}, 1} \ar[d]_{\tau^1_{e, e' \odot x}}  & e' \odot e \odot x  \ar[r]^{1, \tau^1_{e, x}} & e' \odot x \odot e \ar[d]^{1, \tau^2_{x,e}} \\
e' \odot x \odot e \ar[r]_{\tau^2_{e' \odot x, e}} \ar[rru]^{1} & e \odot e' \odot x \ar[r]_{r_{e, e'}, 1} & e' \odot e \odot x
} \]
Two squares commute by the diagrams (\ref{diag-emb-1}).
Then the left diagram of (\ref{diag-SL}) commutes.
Consider the diagram:
\[ \xymatrix@=4ex{
& e \odot e' \odot x \ar[r]^{1, \tau^1_{e', x}} \ar[d]_{\tau^1_{e \otimes e', x}} & e \odot x \odot e' \ar[r]^{1, \tau^2_{x, e'}} \ar[ld]_(0.55){\tau^1_{e, x}, 1} \ar[d]^{\tau^1_{e, x \odot e'}}   & e \odot e' \odot x  \ar[dd]^{\tau^1_{e, e' \odot x}} \\
& x \odot e \odot e'  \ar[r]^{1, r_{e, e'}} \ar[d]_{\tau^2_{x, e \otimes e'}} & x \odot e' \odot e  \ar[rd]^{\tau^2_{x, e'}, 1} \ar[d]_{\tau^2_{x, e' \otimes e}} &  \\
& e \odot e' \odot x & e' \odot e \odot x \ar[l]^{r_{e', e},1} & e' \odot x \odot e \ar[l]^{1, \tau^2_{x,e}} \ar@/^1.7pc/[ll]^{\tau^2_{e' \odot x, e}}
} \]
Four triangles commute by the diagrams (\ref{diag-emb-1}), (\ref{diag-emb-2}) and (\ref{diag-eb-4}).
Two squares commute by the naturalities of $\tau^1$ and $\tau^2$.
\end{proof}

The diagram (\ref{diag-SL}) is the condition of the $\E$-module braiding defined by \cite[Def.\,3.5.1]{SL}.

\begin{expl}
\label{expl-double-braiding}
Let $\C$ be a braided fusion category equipped with a braided functor $T_{\C}: \E \to \C$. Then $\C$ is a braided $\E$-module. The $\E$-module braiding on $\C$ is given by
$\tau_{e, x}: T_{\C}(e) \otimes x \xrightarrow{c_{T_{\C}(e), x}} x \otimes T_{\C}(e) \xrightarrow{c_{x, T_{\C}(e)}} T_{\C}(e) \otimes x$, 
where $c$ is the braiding of $\C$.
\end{expl}

\subsection{Tensor product of braided $\E$-modules}

\begin{defn}
\label{defn-tensorp-braidedE-modules}
Let $(\C, \tau^{\C}), (\D, \tau^{\D}), (\W, \tau^{\W})$ be braided $\E$-modules. 
\emph{A braiding-preserved $\E$-bilinear bifunctor} $F: \C \times \D \to \W$ consists of the following data.
\begin{itemize}
\item $F: \C \times \D \to \W$ is an $\E$-bilinear bifunctor.
That is, for each $d \in \D$, $(F(-, d), s^{F1}): \C \to \W$ is a left $\E$-module functor, where 
\[ s^{F1}_{e, c}: F(e \odot c, d) \simeq e \odot F(c, d), \quad\forall e \in \E, c \in \C \]
is a natural isomorphism. For each $g: d \to d'$ in $\D$, $F(-, g): F(-, d) \Rightarrow F(-, d')$ is a left $\E$-module natural transformation.
For each $c \in \C$, $(F(c, -), s^{F2}): \D \to \W$ is a left $\E$-module functor, where 
\[ s^{F2}_{e, d}: F(c, e \odot d) \simeq e \odot F(c, d), \quad \forall e \in \E, d \in \D \]
is a natural isomorphism.
For each $f: c \to c'$ in $\C$, $F(f, -): F(c, -) \Rightarrow F(c', -)$ is a left $\E$-module natural transformation.

\item $F: \C \times \D \to \W$ is a balanced $\E$-module functor, where the balanced $\E$-module structure on $F$ is defined as
\[ b_{c,e,d}: F(c \odot e, d) \xrightarrow{\tau^{\C2}_{c,e}, 1} F(e \odot c, d) \xrightarrow{s^{F1}_{e, c}} e \odot F(c, d) \xrightarrow{(s^{F2}_{e,d})^{-1}} F(c, e \odot d) \]
for $c \in \C, e \in \E, d \in \D$.
\end{itemize}
such that the following diagrams commute for $e \in \E, c \in \C, d \in \D$:
\begin{equation}
\label{eq-braiding-preserved}
\begin{split}
 \xymatrix{
F(e \odot c, d) \ar[r]^{s^{F1}_{e, c}} \ar[d]_{\tau^{\C1}_{e,c}, 1} & e \odot F(c, d) \ar[r]^{(s^{F2}_{e,d})^{-1}} \ar[d]^{\tau^{\W1}_{e, F(c,d)}} & F(c, e \odot d) \ar[d]^{1, \tau^{\D1}_{e,d}}  \\
F(c \odot e, d) \ar[d]_{\tau^{\C2}_{c,e}, 1} & F(c, d) \odot e \ar[d]^{\tau^{\W2}_{F(c,d), e}}  & F(c, d \odot e) \ar[d]^{1, \tau^{\D2}_{d,e}}  \\
F(e \odot c, d) \ar[r]_{s^{F1}_{e,c}} & e \odot F(c, d) \ar[r]_{(s^{F2}_{e,d})^{-1}} & F(c, e \odot d)
} 
\end{split}
\end{equation}
\end{defn}

The category $\Fun^{\bal}_{\E}(\C, \D; \W)$ consists of the following data.
\begin{itemize}
\item Its objects are braiding-preserved $\E$-bilinear bifunctors $\C \times \D \to \W$.
\item A morphism between two braiding-preserved $\E$-bilinear bifunctor $F, G: \C \times \D \rightrightarrows \W$ is a natural transformation $\alpha: F \Rightarrow G$ such that the two diagrams 
    \[ 
    \vcenter{\xymatrix{
    F(c \odot e, d) \ar[r]^{b^F_{c,e,d}} \ar[d]_{\alpha_{c \odot e, d}} & F(c, e \odot d) \ar[d]^{\alpha_{c, e \odot d}} \\
    G(c \odot e, d) \ar[r]_{b^G_{c,e,d}} & G(c, e \odot d)
    }}
    \qquad 
    \vcenter{\xymatrix{
    F(e \odot c, d) \ar[r]^{s^{F1}_{e, c}} \ar[d]_{\alpha_{e \odot c, d}} & e \odot F(c, d) \ar[d]^{1, \alpha_{c, d}} \\
    G(e \odot c, d) \ar[r]_{s^{G1}_{e, c}} & e \odot G(c, d)
    }} \]
commute for $c \in \C, e \in \E, d \in \D$, where $b^F$ and $b^G$ are the balanced $\E$-module structures on $F$ and $G$ respectively. A natural transformation $\alpha: F \Rightarrow G$ is called a balanced $\E$-module natural transformation if $\alpha$ satisfies the left diagram. For each $d \in \D$, $\alpha: (F(-, d), s^{F1}) \Rightarrow (G(-, d), s^{G1})$ is called a left $\E$-module natural transformation if $\alpha$ satisfies the right diagram.
\end{itemize}

\begin{prop}
For each $c \in \C$, $\alpha: (F(c, -), s^{F2}) \Rightarrow (G(c, -), s^{G2})$ is a left $\E$-module natural transformation.
\end{prop}
\begin{proof}
Consider the diagram:
\[ \xymatrix{
F(c \odot e, d) \ar[r]^{\tau^{\C2}_{c,e}, 1} \ar[d]_{\alpha_{c \odot e, d}} & F(e \odot c, d) \ar[r]^{s^{F1}_{e, c}} \ar[d]_{\alpha_{e \odot c, d}} & e \odot F(c, d) \ar[r]^{(s^{F2}_{e, d})^{-1}} \ar[d]_{1, \alpha_{c, d}} & F(c, e \odot d) \ar[d]^{\alpha_{c, e \odot d}} \\
G(c \odot e, d) \ar[r]_{\tau^{\C2}_{c,e}, 1} & G(e \odot c, d) \ar[r]_{s^{G1}_{e, c}} & e \odot G(c, d) \ar[r]_{(s^{G2}_{e, d})^{-1}} & G(c, e \odot d)
} \]
The outward diagram commutes since $\alpha$ is a balanced $\E$-module natural transformation. The first square commutes by the naturality of $\alpha$. The second square commutes since $\alpha: (F(-, d), s^{F1}) \Rightarrow (G(-, d), s^{G1})$ is a left $\E$-module natural transformation. Then the third square commutes.
\end{proof}

\begin{prop}
For $e_1, e_2 \in \E$, $c \in \C, d \in \D$, the following diagram commutes
\begin{equation}
\label{diag-sf1-sf2}
\begin{split}
 \xymatrix{
F(e_1 \odot c, e_2 \odot d) \ar[r]^{s^{F2}_{e_2, d}} \ar[d]_{s^{F1}_{e_1, c}} & e_2 \odot F(e_1 \odot c, d) \ar[r]^{1, s^{F1}_{e_1, c}} & e_2 \odot e_1 \odot F(c, d) \ar[ld]^{r_{e_2, e_1}, 1} \\
e_1 \odot F(c, e_2 \odot d) \ar[r]_{1, s^{F2}_{e_2, d}} & e_1 \odot e_2 \odot F(c, d) &
} 
\end{split}
\end{equation}
\end{prop}

\begin{proof}
Since $F: \C \times \D \to \W$ is a balanced $\E$-module functor, the following outward diagram commutes
\[ \xymatrix{
F(c \odot (e_1 \otimes e_2), d) \ar[r]^{\tau^{\C2}_{c, e_1 \otimes e_2}, 1} \ar[d]_{\tau^{\C2}_{c \odot e_1, e_2}} & F((e_1 \otimes e_2) \odot c, d) \ar[r]^{s^{F1}_{e_1 \otimes e_2, c}} \ar[d]^{r_{e_1, e_2}, 1} & (e_1 \otimes e_2) \odot F(c, d) \ar[d]^{r_{e_1, e_2}} \ar[r]^{(s^{F2}_{e_1 \otimes e_2, d})^{-1}} & F(c, (e_1 \otimes e_2) \odot d) \ar[ddd]^{s^{F2}_{e_1, e_2 \odot d}}  \\
F(e_2 \odot c \odot e_1, d) \ar[r]_{\tau^{\C2}_{c, e_1}, 1} \ar[d]_{s^{F1}_{e_2, c \odot e_1}} & F(e_2 \odot e_1 \odot c, d) \ar[r]^{s^{F1}_{e_2 \otimes e_1, c}} \ar[d]^{s^{F1}_{e_2, e_1 \odot c}} & e_2 \odot e_1 \odot F(c, d) &   \\
e_2 \odot F(c \odot e_1, d) \ar[d]_{(s^{F2}_{e_2, d})^{-1}} & e_2 \odot F(e_1 \odot c, d) \ar[d]^{(s^{F2}_{e_2, d})^{-1}} \ar[ru]_{1, s^{F1}_{e_1, c}} & & \\
F(c \odot e_1, e_2 \odot d) \ar[r]_{\tau^{\C2}_{c, e_1},1} & F(e_1 \odot c, e_2 \odot d) \ar[rr]_{s^{F1}_{e_1, c}} & & e_1 \odot F(c, e_2 \odot d)   \ar@/_1pc/[luuu]_{1, s^{F2}_{e_2, d}}
} \]
The two triangles commute since $(F(-,d), s^{F1}): \C \to \W$ and $(F(c,-), s^{F2}): \D \to \W$ are left $\E$-module functors. The left-bottom hexagon and middle-top square commute by the naturality of $s^{F1}$ and the left $\E$-module natural isomorphism $F(c \odot e_1, -) \Rightarrow F(e_1 \odot c, -)$ for $\tau^{\C_2}_{c, e_1}: c \odot e_1 \to e_1 \odot c$ in $\C$. 
The left-top square commutes by the commutativity of the diagrams
\[ \xymatrix@=3ex{
F(c \odot (e_1 \otimes e_2), d) \ar[rr]^{\tau^{\C2}_{c, e_1 \otimes e_2}, 1} \ar[dd]_{\tau^{\C2}_{c \odot e_1, e_2}} \ar[rd]^{r_{e_1, e_2}, 1} & & F((e_1 \otimes e_2) \odot c, d) \ar[dd]^{r_{e_1, e_2}, 1}  \\
& F(c \odot e_2 \odot e_1, d) \ar[ld]^{\tau^{\C2}_{c, e_2}, 1} \ar[rd]^{\tau^{\C2}_{c, e_2 \otimes e_1}} & \\
F(e_2 \odot c \odot e_1, d) \ar[rr]_{\tau^{\C2}_{c, e_1},1} & & F(e_2 \odot e_1 \odot c, d)
} \]
The two triangles commute by the diagrams (\ref{diag-emb-2}) and (\ref{diag-eb-4}). The square commutes by the naturality of $\tau^{\C2}$. 
Then we complete the proof.
\end{proof}

\begin{prop}
Let $\hat{b}_{c, e, d}: F(e \odot c, d) \xrightarrow{s^{F1}_{e, c},1} e \odot F(c, d) \xrightarrow{(s^{F2}_{e, d})^{-1}} F(c, e \odot d)$ for $e \in \E$, $c \in \C$, $d \in \D$.
The following diagram commutes for $e_1, e_2 \in \E$:
\begin{equation}
\label{diag-b-b}
\begin{split}
 \xymatrix{
F((e_1 \odot c) \odot e_2, d) \ar[r]^{b_{e_1 \odot c, e_2, d}} \ar[d] & F(e_1 \odot c, e_2 \odot d) \ar[r]^{\hat{b}_{c, e_1, e_2 \odot d}} &  F(c, e_1 \odot e_2 \odot d) \ar[d]^{1, r_{e_1, e_2}} \\
F(e_1 \odot (c \odot e_2), d) \ar[r]_{\hat{b}_{c \odot e_2, e_1, d}} & F(c \odot e_2, e_1 \odot d) \ar[r]_{b_{c, e_2, e_1 \odot d}} & F(c, e_2 \odot e_1 \odot d)
} 
\end{split}
\end{equation}
\end{prop}

\begin{proof}
We want to prove the following outward diagram commutes:
\[ \xymatrix@=5ex{
F((e_1 \odot c) \odot e_2, d) \ar[r]  \ar[d]_{\tau^{\C2}_{e_1 \odot c, e_2}, 1}  & F(e_1 \odot (c \odot e_2), d) \ar[r]^{\hat{b}_{c \odot e_2, e_1, d}}\ar[d]_{\tau^{\C2}_{c, e_2},1} & F(c \odot e_2, e_1 \odot d) \ar[d]^{\tau^{\C2}_{c,e_2},1} \\
F(e_2 \odot e_1 \odot c, d) \ar[r]^{r_{e_2,e_1},1} \ar[dd]_{\hat{b}_{e_1 \odot c, e_2, d}} \ar[rd]_{\hat{b}_{c,e_2 \otimes e_1, d}} & F(e_1 \odot e_2 \odot c, d) \ar[r]^{\hat{b}_{e_2 \odot c, e_1, d}} \ar@/^4pc/[dd]^{\hat{b}_{c, e_1 \otimes e_2, d}} & F(e_2 \odot c, e_1 \odot d) \ar[dd]^{\hat{b}_{c, e_2, e_1 \odot d}} \\
& F(c, e_2 \odot e_1 \odot d) \ar[d]_{1,r_{e_2, e_1}} & \\
F(e_1 \odot c, e_2 \odot d) \ar[r]_{\hat{b}_{c, e_1, e_2 \odot d}} & F(c, e_1 \odot e_2 \odot d) \ar[r]_{1, r_{e_1, e_2}} & F(c, e_2 \odot e_1 \odot d) 
} \]
The upper-left square commutes by the diagram (\ref{diag-emb-1}).
The upper-right square commutes by the naturality of $s^{F1}$ and the left $\E$-module natural isomorphism $F(c \odot e_2, -) \Rightarrow F(e_2 \odot c, -)$ for $\tau^{\C2}_{c, e_2}: c \odot e_2 \to e_2 \odot c$ in $\C$.The middle square commutes by the naturalities of $s^{F1}$ and $s^{F2}$.
The left-down square and the right-down square commute by the following diagram:
\[ \xymatrix@=5ex{
F(e_1 \odot e_2 \odot c, d) \ar[r]^{s^{F1}_{e_1 \otimes e_2, c}} \ar[dd]_{s^{F1}_{e_1, e_2 \odot c}} & e_1 \odot e_2 \odot F(c, d) \ar[r]^{(s^{F2}_{e_1 \otimes e_2, d})^{-1}} \ar[d]^{r_{e_1, e_2},1} & F(c, e_1 \odot e_2 \odot d) \ar[d]^{1,r_{e_1,e_2}} \\
  & e_2 \odot e_1 \odot F(c, d) \ar[r]^{(s^{F2}_{e_2 \otimes e_1, d})^{-1}} & F(c, e_2 \odot e_1 \odot d) \\
e_1 \odot F(e_2 \odot c, d) \ar[r]_{(s^{F2}_{e_1, d})^{-1}} \ar@/^1pc/[ruu]^{1,s^{F1}_{e_2,c}}  & F(e_2 \odot c, e_1 \odot d) \ar[r]_{s^{F1}_{e_2,c}} & e_2 \odot F(c, e_1 \odot d) \ar[u]_{(s^{F2}_{e_2, e_1 \odot d})^{-1}} \ar[lu]^(0.6){1,s^{F2}_{e_1, d}}
} \]
The two triangles commute by the left $\E$-module functors $(F(-, d), s^{F1}): \C \to \W$ and $(F(c, -), s^{F2}): \D \to \W$. 
The square commutes by the naturality of $s^{F2}$.
The pentagon commutes by the diagram (\ref{diag-sf1-sf2}).
\end{proof}

\begin{prop}
The balanced $\E$-module structure $b_{c,e,d}$ on $F: \C \times \D \to \W$ is a left $\E$-module natural isomorphism. That is, the diagrams 
\begin{equation}
\label{matrix-tp-bem}
\vcenter{\xymatrix{
F(e' \odot (c \odot e), d) \ar[r]^{s^{F1}_{e', c \odot e}} \ar[d] & e' \odot F(c \odot e, d) \ar[dd]^{1, b_{c,e,d}} \\
F((e' \odot c) \odot e, d) \ar[d]_{b_{e' \odot c, e, d}} & \\
F(e' \odot c, e \odot d) \ar[r]_{s^{F1}_{e', c}} & e' \odot F(c, e \odot d)
}}
\qquad
\vcenter{\xymatrix{
 F(c \odot e, e' \odot d) \ar[r]^{s^{F2}_{e', d}} \ar[d]_{b_{c, e, e' \odot d}} & e' \odot F(c \odot e, d) \ar[dd]^{1, b_{c, e, d}} \\
F(c, e \odot e' \odot d)  \ar[d]_{1, r_{e, e'}} & \\
F(c, e' \odot e \odot d) \ar[r]_{s^{F2}_{e', e \odot d}}  & e' \odot F(c, e \odot d)
}}
\end{equation}
commute for $e' \in \E$.
For each $c \in \C, d \in \D$, we define right $\E$-module functors $(F(-, d), t^{F1}): \C \to \W$ and $(F(c,-), t^{F2}): \D \to \W$, where
\[  t^{F1}_{c, e}: F(c \odot e, d) \xrightarrow{(\tau^{\C1}_{e, c})^{-1}, 1} F(e \odot c, d) \xrightarrow{s^{F1}_{e, c}} e \odot F(c, d) \xrightarrow{\tau^{\W1}_{e, F(c,d)}} F(c, d) \odot e, \quad \forall e \in \E, c \in \C  \]
\[  t^{F2}_{d, e}: F(c, d \odot e) \xrightarrow{1, \tau^{\D2}_{d,e}} F(c, e \odot d) \xrightarrow{s^{F2}_{e,d}} e \odot F(c,d) \xrightarrow{(\tau^{\W2}_{F(c,d), e})^{-1}} F(c, d) \odot e, \quad \forall e \in \E, d \in \D  \]
Then the following diagrams commute:
\begin{equation}
\label{diag-tF12}
\begin{split}
\vcenter{\xymatrix{
F(c \odot e \odot e', d) \ar[r]^{t^{F1}_{c \odot e, e'}} \ar[d]_{r_{e, e'}, 1} & F(c \odot e, d) \odot e'  \ar[dd]^{b_{c, e, d}, 1} \\
F(c \odot e' \odot e, d) \ar[d]_{b_{c \odot e', e, d}}  & \\
F(c \odot e', e \odot d) \ar[r]_{t^{F1}_{c, e'}} & F(c, e \odot d) \odot e' 
}} 
\qquad
\vcenter{\xymatrix{
F(c \odot e,  d \odot e') \ar[r]^{t^{F2}_{d, e'}} \ar[d]^{b_{c, e, d \odot e'}} & F(c \odot e, d) \odot e' \ar[dd]^{b_{c, e, d}, 1} \\
F(c, e \odot (d \odot e')) \ar[d] & \\
F(c, (e \odot d) \odot e') \ar[r]_{t^{F2}_{e \odot d, e'}} & F(c, e \odot d) \odot e'
}} 
\end{split}
\end{equation}
\end{prop}

\begin{proof}
Proof the left diagram of (\ref{matrix-tp-bem}).  
By the definition of $b_{c,e,d}$, we want to prove the following outward diagram commutes:
\[ \xymatrix{
F(e' \odot (c \odot e), d) \ar[rr]^{s^{F1}_{e', c \odot e}} \ar[d] \ar[rd]^{\tau^{\C2}_{c,e}, 1} & & e' \odot F(c \odot e, d) \ar[d]^{1,\tau^{\C2}_{c, e}} \\
F((e' \odot c) \odot e, d) \ar[d]_{\tau^{\C2}_{e' \odot c, e}, 1} & F(e' \odot e \odot c, d) \ar[ld]_{r_{e', e}, 1} \ar[r]^{s^{F1}_{e', e \odot c}} \ar[rd]_(0.4){s^{F1}_{e' \otimes e, c}} & e' \odot F(e \odot c, d) \ar[d]^{1,s^{F1}_{e, c}} \\
F(e \odot (e' \odot c), d) \ar[d]_{s^{F1}_{e, e' \odot c}} \ar[r]^{s^{F1}_{e \otimes e', c}} & e \odot e' \odot F(c, d) \ar[r]_{r_{e, e'}, 1} & e' \odot e \odot F(c, d) \ar[d]^{1,(s^{F2}_{e, d})^{-1}} \\
e \odot F(e' \odot c, d) \ar[ru]_(0.6){s^{F1}_{e', c}} \ar[r]_{1, (s^{F2}_{e, d})^{-1}} & F(e' \odot c, e \odot d) \ar[r]_{s^{F1}_{e', c}} & e' \odot F(c, e \odot d)
} \]
The left-upper square commutes by the diagram (\ref{diag-emb-1}).
The two triangles commute by the left $\E$-module functor $(F(-, d), s^{F1}): \C \to \W$. The upper square and the middle square commute by the naturality of $s^{F1}$.
The pentagon commutes by the diagram (\ref{diag-sf1-sf2}).
Then the outward diagram commutes.

Proof the right diagram of (\ref{matrix-tp-bem}). Consider the diagram:
\[ \xymatrix{
F((e' \odot c) \odot e, d) \ar[d] \ar[r]^{b_{e' \odot c, e, d}} & F(e' \odot c, e \odot d) \ar[r]^{s^{F1}_{e', c}} & e' \odot F(c, e \odot d) \ar[r]^{1, (s^{F2}_{e', e \odot d})^{-1}} & F(c, e' \odot e \odot d) \ar[d]^{1, r_{e', e}} \\
F(e' \odot (c \odot e), d) \ar[r]_{s^{F1}_{e', c \odot e}} & e' \odot F(c \odot e, d) \ar[r]_{(s^{F2}_{e', d})^{-1}} \ar[ru]_(0.6){1, b_{c, e, d}} & F(c \odot e, e' \odot d) \ar[r]_{b_{c, e, e' \odot d}} & F(c, e \odot e' \odot d)
} \]
The outward diagram commutes by the diagram (\ref{diag-b-b}).
The left pentagon commutes by the left diagram of (\ref{matrix-tp-bem}). Then the right pentagon commutes.

Proof the left diagram of (\ref{diag-tF12}). By the definition of $t^{F1}$, we want to prove the following outward diagram commutes:
\[ 
\xymatrixrowsep{1.2pc}
\xymatrix{
F(c \odot e \odot e', d) \ar[r]^{(\tau^{\C1}_{e', c \odot e})^{-1},1} \ar[d]_{r_{e,e'},1} & F(e' \odot (c \odot e), d) \ar[r]^{s^{F1}_{e',c \odot e}} \ar[d] & e' \odot F(c \odot e, d) \ar[r]^{\tau^{\W1}_{e', F(c \odot e, d)}} \ar[dd]^{1, b_{c,e,d}} & F(c \odot e, d) \odot e' \ar[dd]^{b_{c,e,d},1} \\
F(c \odot e' \odot e, d) \ar[d]_{b_{c \odot e', e, d}} & F((e' \odot c) \odot e, d) \ar[l]_{\tau^{\C1}_{e',c},1} \ar[d]^{b_{e' \odot c, e, d}} & & \\
F(c \odot e', e \odot d) \ar[r]_{(\tau^{\C1}_{e',c})^{-1},1} & F(e' \odot c, e \odot d) \ar[r]_{s^{F1}_{e',c}} & e' \odot F(c, e \odot d) \ar[r]_{\tau^{\W1}_{e', F(c, e \odot d)}} & F(c, e \odot d) \odot e'
} \]
The left-upper square commutes by the diagram (\ref{diag-eb-4}). 
The middle pentagon commutes by the diagram (\ref{matrix-tp-bem}).
The left-bottom and right-most squares commute by the naturalities of $b$ and $\tau^{\W1}$.

Proof the right diagram of (\ref{diag-tF12}). By the definition of $t^{F2}$, we want to prove the following outward diagram commutes:
\[ 
\xymatrixrowsep{1.2pc}
\xymatrix{
F(c \odot e, d \odot e') \ar[r]^{1, \tau^{\D2}_{d,e'}} \ar[d]_{b_{c, e, d \odot e'}} & F(c \odot e, e' \odot d) \ar[d]^{b_{c,e, e' \odot d}} \ar[r]^{s^{F2}_{e', d}} & e' \odot F(c \odot e, d) \ar[dd]^{1, b_{c,e,d}} \ar[r]^{(\tau^{\W2}_{F(c \odot e, d), e'})^{-1}} & F(c \odot e, d) \odot e' \ar[dd]^{b_{c,e,d},1}  \\
F(c, e \odot (d \odot e')) \ar[r]^{1, \tau^{\D2}_{d,e'}} \ar[d] & F(c, e \odot e' \odot d) \ar[d]^{1, r_{e,e'}} & & \\
F(c, (e \odot d) \odot e') \ar[r]_{1, \tau^{\D2}_{e \odot d, e'}} & F(c, e' \odot (e \odot d)) \ar[r]_{s^{F2}_{e', e \odot d}} & e' \odot F(c, e \odot d) \ar[r]_{(\tau^{\W2}_{F(c, e \odot d), e'})^{-1}} & F(c, e \odot d) \odot e'
} \]
The left-upper and right-most squares commute by the naturalities of $b$ and $\tau^{\W2}$. The middle pentagon commutes by the diagram (\ref{matrix-tp-bem}).
The left-bottom square commutes by the diagram (\ref{diag-emb-1}).
\end{proof}

\begin{defn}[\cite{SL}\,Def.\,3.6.1]
Let $(\C, \tau^{\C})$ and $(\D, \tau^{\D})$ be finite braided $\E$-modules.
The \emph{relative tensor product} $\C \boxdot_{\E} \D$ is defined as the full subcategory of $\C \boe \D$ formed by the objects $x$ such that $\tau^{\C}_{e, x} = \tau^{\D}_{e, x}$ for all $e \in \E$.
\end{defn}

\begin{rem}
For $e \in \E$, $c \boe d \in \C \boe \D$,
the left and right $\E$-module structure on $\C \boe \D$ are defined as $e \odot (c \boe d) = (e \odot c) \boe d$ and $(c \boe d) \odot e = c \boe (d \odot e)$.
Two $\E$-module braidings $\tau^{\C}_{e, c \boe d}$ and $\tau^{\D}_{e, c \boe d}$ on $\C \boe \D$ are induced by:
\[ \tau^{\C}_{e, c \boe d}:  (e \odot c) \boe d \xrightarrow{\tau^{\C}_{e, c} \boe \id_d} (e \odot c) \boe d \]
\[ \tau^{\D}_{e, c \boe d}: (e \odot c) \boe d \xrightarrow{(\tau^{\C2}_{c, e})^{-1},1} (c \odot e) \boe d \xrightarrow{b_{c,e,d}} c \boe (e \odot d) \xrightarrow{1, \tau^{\D}_{e,d}} c \boe (e \odot d) \xrightarrow{b^{-1}_{c,e,d}} (c \odot e) \boe d \xrightarrow{\tau^{\C2}_{c,e},1} (e \odot c) \boe d \]
If $\tau^{\C}_{e, c \boe d} = \tau^{\D}_{e, c \boe d}$, the following diagram commutes
\begin{equation}
\label{eq-defn-bEm}
\begin{split}
 \xymatrix{
 (c \odot e) \boe d \ar[d]_{b_{c,e,d}} \ar[r]^{\tau^{\C2}_{c,e}, 1} &(e \odot c) \boe d  \ar[r]^{\tau^{\C1}_{e,c}, 1} & (c \odot e) \boe d \ar[d]^{b_{c,e,d}}  \\
 c \boe (e \odot d) \ar[r]_{1, \tau^{\D1}_{e, d}} & c \boe (d \odot e) \ar[r]_{1, \tau^{\D2}_{d,e}}& c \boe (e \odot d) 
}
\end{split}
\end{equation}
The $\E$-bilinear structure on $\boe: \C \times \D \to \C \boe \D$ is defined by $(s^{F1}_{e, c})^{-1}: e \odot (c \boe d) = (e \odot c) \boe d$ and
$(s^{F2}_{e,d})^{-1}: e \odot (c \boe d) = (e \odot c) \boe d \xrightarrow{(\tau^{\C2}_{c, e})^{-1}, 1} (c \odot e) \boe d \xrightarrow{b_{c, e, d}} c \boe (e \odot d)$.
$c \boe d$ satisfies the diagram (\ref{eq-braiding-preserved}) if and only if $c \boe d$ satisfies the diagram (\ref{eq-defn-bEm}).
If $\tau^{\C}_{e,c} = \id$ and $\tau^{\D}_{e, d} = \id$, then $\C \boxdot_{\E} \D = \C \boe \D$.
\end{rem}

\begin{defn}[\cite{SL}\,Rem.\,3.6.7]
Let $(\C, \tau^{\C})$ and $(\D, \tau^{\D})$ be braided $\E$-modules. The \emph{relative tensor product of $\C$ and $\D$} is a braided $\E$-module $\C \boxdot_{\E} \D$, together with a braiding preserved $\E$-bilinear bifunctor 
$\boxdot_{\E}: \C \times \D \to \C \boxdot_{\E} \D$, such that for every braided $\E$-module $\W$, composition with $\boxdot_{\E}$ induces an equivalence of categories $\Fun'_{\E}(\C \boxdot_{\E} \D, \W) \simeq \Fun^{\bal}_{\E}(\C, \D; \W)$.
\end{defn}

\begin{rem}
The universal property of the relative tensor product is illustrated in the following commutative diagram:
\[ \xymatrix{
\C \times \D \ar[r]^{\boxdot_{\E}} \ar[rd]_{F} &  \C \boxdot_{\E} \D \ar[d]^{\exists ! \underline{F}} \\
& \W
} \]
for $F \in \Fun^{\bal}_{\E}(\C, \D; \W)$, where $\underline{F} \in \Fun'_{\E}(\C \boxdot_{\E} \D, \W)$.
\end{rem}

\begin{expl}
\label{expl-bE-mo}
Let $\A$ be a subcategory of a braided fusion category $\C$.
The \emph{centralizer of $\A$ in $\C$}, denoted by $\A'|_{\C}$, is the full subcategory of objects $x \in \C$ such that $c_{a, x} \circ c_{x, a} = \id_{x \otimes a}$ for all $a \in \A$, where $c$ is the braiding of $\C$.

Let $\B$ be a braided $\E$-module. 
$\E \boxdot_{\E} \B$ is the full subcategory of $\B$ consisting of the objects $x$ such that $\tau^{\B}_{e,x} = \id_{e \odot x}$ for $e \in \E, x \in \B$. 
In particular, when $\B$ is a braided fusion category with a fully faithful braided functor $\E \to \B$, $\E \boxdot_{\E} \B = \E'|_{\B}$ by Expl.\,\ref{expl-double-braiding}.
\end{expl}

\begin{rem}
\label{rem-sun-prop-bE}
Let $\B, \C, \D$ be braided $\E$-modules. The equivalences 
$(\B \boxdot_{\E} \C) \boxdot_{\E} \D \simeq \B \boxdot_{\E} (\C \boxdot_{\E} \D)$, $\C \boxdot_{\E} \D \simeq \D \boxdot_{\E} \C$, $Z(\E) \boxdot_{\E} \C \simeq \C$, and $(\E \boxdot_{\E} \C) \boe (\E \boxdot_{\E} \D) \simeq \E \boxdot_{\E} (\C \boxdot_{\E} \D)$ hold as braided $\E$-modules by \cite[Rem.\,3.6.9, Prop.\,3.6.12, Prop.\,3.6.12]{SL}.
If $\C \simeq \D$ as braided $\E$-modules, $\E \boxdot_{\E} \C \simeq \E \boxdot_{\E} \D$.
\end{rem}

\subsection{Modular extensions}

The following definitions, propositions and lemmas come from \cite[Def.\,3.4.1, Prop.\,3.7.1]{SL} and \cite[Def.\,4.4, Lem.\,4.11, Lem.\,4.16, Thm.\,4.20]{TL}.
\begin{defn}
A \emph{braided multifusion category containing $\E$} is a braided multifusion category $\C$ equipped with a $\bk$-linear braided monoidal functor $\phi_{\C}: \E \to \C$. A braided fusion category $\C$ is fully faithful containing $\E$ if $\phi_{\C}$ is fully faithful.
\end{defn}

Let $\C$ be a braided fusion category with the braiding $c$ and $A$ a connected \'{e}tale algebra in $\C$. 
A right $A$-module $(M, \mu_M: M \otimes A \to M)$ is called \emph{local} if $\mu_M \circ c_{A, M} \circ c_{M, A} = \mu_M$.
The category $\C_A$ denotes the category of right $A$-modules.
The category $\C^0_A$ denotes the full subcategory of $\C_A$ formed by local modules.

Let $\C, \D$ be braided fusion categories fully faithful containing $\E$.
Let $R: \E \to \E \boxtimes \E$ be the right adjoint functor of the tensor product functor $\otimes: \E \boxtimes \E \to \E$. Then $L_{\E} \coloneqq R(\unit_{\E})$ is a connected \'{e}tale algebra in $\E \boxtimes \E$. Let $\mO(\E)$ denote the set of simple objects in $\E$. $L_{\E}$ has a decomposition $L_{\E} = \oplus_{i \in \mO(\E)} i^L \boxtimes i$. 

\begin{prop}[\cite{SL}\,Prop.\,3.7.1]
\label{prop-3.7.1}
The equivalence $(\C \boxtimes \D)_{L_{\E}} \cong \C \boe \D$ restricts to an equivalence $(\C \boxtimes \D)^0_{L_{\E}} \cong \C \boxdot_{\E} \D$.
\end{prop}

\begin{defn}
Let $\C$ be a non-degenerate braided fusion category over $\E$ ($\NBFCe$, see Def.\,\ref{defn-braided-over-E}). A modular extension of $\C$ is a pair $(\M, \iota_{\M})$, where $\M$ is a braided fusion category and $\iota_{\M}: \C \hookrightarrow \M$ is a $\bk$-linear braided full embedding, such that $\M' = \vect$ and $\E \boxdot_{\E} \M = \C$.
\end{defn}

We have the equation $\E \boxdot_{\E} \M = \E'|_{\M} = \C$ by Expl.\,\ref{expl-bE-mo}.
The set $\M_{ext}(\C)$ denotes the equivalence classes of the modular extensions of $\C$. 

\begin{lem}
\label{lem-modular-tp}
Let $\C, \D$ be two $\NBFCe$.
If $\M_{ext}(\C)$ and $\M_{ext}(\D)$ are not empty, then $\M_{ext}(\C \boe \D)$ is not empty, and there is a well-defined map
\[ \boe^{(-,-)}: \M_{ext}(\C) \times \M_{ext}(\D) \to \M_{ext}(\C \boe \D). \]
More explicitly, let $(\M, \iota_{\M}: \C \hookrightarrow \M)$ and $(\N, \iota_{\N}: \D \hookrightarrow \N)$ be the modular extensions of $\C$ and $\D$  respectively. Then $\C \boe \D$ is a $\NBFCe$ and the pair
\[ \M \boe^{(\iota_{\M}, \iota_{\N})} \N \coloneqq \big( \M \boxdot_{\E} \N, \iota_{\M} \boe \iota_{\N}: \C \boxdot_{\E} \D \hookrightarrow \M \boxdot_{\E} \N \big) \]
is the modular extension of $\C \boe \D$.
\end{lem} 

\begin{proof}
We have the conclusion by \cite[Lem.\,4.11]{TL}, Prop.\,\ref{prop-3.7.1} and the equivalence
$\E \boxdot_{\E} (\M \boxdot_{\E} \N) \simeq (\E \boxdot_{\E} \M) \boe (\E \boxdot_{\E} \N) = \C \boe \D$.
\end{proof}

Let $Z(\E)$ be the Drinfeld center of $\E$. The objects of $Z(\E)$ are pairs $(e, r_{e,-})$, where $e \in \E$ and $r_{e, -}: e \otimes - \simeq - \otimes e, - \in \E$ is a natural isomorphism satisfying some conditions.
There is a canonical embedding $\iota_0: \E \hookrightarrow Z(\E)$, $e \mapsto (e, r_{e,-})$.
The pair $(Z(\E), \iota_0)$ gives a modular extension of $\E$.
\begin{prop}
The multiplication $\boe^{(-,-)}$ is associative, commutative and unital.
More explicitly, Let $(\CL, \iota_{\CL}: \B \hookrightarrow \CL)$, $(\M, \iota_{\M}: \C \hookrightarrow \M)$ and $(\N, \iota_{\N}: \D \hookrightarrow \N)$ be three modular extensions of the $\NBFCe$'s $\B, \C$ and $\D$ respectively. Then the following equivalences hold:
\[ \CL \boe^{(\iota_{\CL}, \iota_{\M} \boe \iota_{\N})} (\M \boe^{(\iota_{\M},\iota_{\N})} \N) \simeq (\CL \boe^{(\iota_{\CL}, \iota_{\M})} \M) \boe^{(\iota_{\CL} \boe \iota_{\M}, \iota_{\N})} \N \]
\[ \M \boe^{(\iota_{\M}, \iota_{\N})} \N \simeq \N \boe^{(\iota_{\N}, \iota_{\M})} \M,  \quad \quad  Z(\E) \boe^{(\iota_0, \iota_{\M})} \M \simeq (\M,\iota_{\M}). \]
\end{prop}

\begin{proof}
By Rem.\,\ref{rem-sun-prop-bE}, the equivalences $(\CL \boxdot_{\E} \M) \boxdot_{\E} \N \simeq \CL \boxdot_{\E} (\M \boxdot_{\E} \N)$, $\M \boxdot_{\E} \N \simeq \N \boxdot_{\E} \M$ and  $Z(\E) \boxdot_{\E} \M \simeq \M$ hold as braided $\E$-modules.
Then we get $\E \boxdot_{\E} ((\CL \boxdot_{\E} \M) \boxdot_{\E} \N) \simeq (\B \boe \C) \boe \D \simeq \B \boe (\C \boe \D) \simeq \E \boxdot_{\E} (\CL \boxdot_{\E} (\M \boxdot_{\E} \N))$,
$\E \boxdot_{\E} (\M \boxdot_{\E} \N) \simeq \C \boe \D \simeq \D \boe \C \simeq \E \boxdot_{\E} (\N \boxdot_{\E} \M)$
and $\E \boxdot_{\E} (Z(\E) \boxdot_{\E} \M) \simeq \E \boe \C \simeq \C \simeq \E \boxdot_{\E} \M$.
\end{proof}

\begin{thm}[\cite{TL}\,Thm.\,4.20]
The set $\M_{ext}(\E)$, together with the multipilication $\boe^{(-,-)}$ and the identity element $(Z(\E), \iota_0)$, defines a finite group.
\end{thm}
\begin{proof}
By \cite[Rem.\,4.3, Lem.\,4.16]{TL}, $(\overline{\M}, \overline{\iota_{\M}} \coloneqq \iota_{\M}: \overline{\E} \hookrightarrow \overline{\M})$ is a modular extension of $\overline{\E}$ and
$\M \boe^{(\iota_{\M}, \overline{\iota_{\M}})} \overline{\M} \coloneqq \big( \M \boxdot_{\E} \overline{\M}, \E \hookrightarrow \M \boxdot_{\E} \overline{\M} \big) \simeq (Z(\E), \iota_0).$
\end{proof}

\section{Functoriality of the center}
\subsection{Modules over a multifusion category over $\E$}

\begin{defn}[\cite{DNO}\,Def.\,2.5,\,2.7]
\label{defn-oE}
\emph{A multifusion category over $\E$} is a multifusion category $\A$ equipped with a $\bk$-linear braided monoidal functor $T_{\A}: \E \to Z(\A)$.
A multifusion category over $\E$ is fully faithful if the functor $\E \to Z(\A)$ is fully faithful.
A fusion category over $\E$ is fully faithful if the central functor $\E \to Z(\A) \xrightarrow{\forget} \A$ is fully faithful, where $\forget$ is the forgetful functor.

A \emph{monoidal functor over $\E$} between two multifusion categories $\A$, $\B$ over $\E$ is a $\bk$-linear monoidal functor $F: \A \rightarrow \B$ equipped with a monoidal natural isomorphism $u_e: F(T_{\A}(e)) \rightarrow T_{\B}(e)$ in $\B$ for each $e \in \E$, called the structure of monoidal functor over $\E$ on $F$, such that the diagram 
\begin{equation}
\begin{split}
\xymatrix{
F(T_{\A}(e) \otimes x) \ar[r]^{J_{T_{\A}(e), x}} \ar[d]_{F(z_{e, x})} & F(T_{\A}(e)) \otimes F(x) \ar[r]^{u_e, 1} & T_{\B}(e) \otimes F(x) \ar[d]^{\hat{z}_{e, F(x)}}\\
F(x \otimes T_{\A}(e)) \ar[r]^{J_{x, T_{\A}(e)}} & F(x) \otimes F(T_{\A}(e)) \ar[r]^{1, u_e} & F(x) \otimes T_{\B}(e)
}
\end{split}
\end{equation}
commutes for $e \in \E, x \in \A$, $(T_{\A}(e), z_{e,-}) \in Z(\A)$ and $(T_{\B}(e), \hat{z}_{e,-}) \in Z(\B)$, where $J$ is the monoidal structure of $F$.
\end{defn}

\begin{expl}
\label{ex3}
Let $\C$ and $\D$ be multifusion categories over $\E$.
$\C \boe \D$ is a multifusion category over $\E$. 
We define a monoidal functor $T_{\C \boe \D}: \E \simeq \E \boe \E \xrightarrow{T_{\C} \boe T_{\D}} \C \boe \D$ by $e \mapsto e \boe \unit_{\E} \mapsto T_{\C}(e) \boe T_{\D}(\unit_{\E}) = T_{\C}(e) \boe \unit_{\D}$ for $e \in \E$.
 And the central structure $\sigma$ on $T_{\C \boe \D}$ is induced by
 \[ \sigma_{e, c \boe d}: T_{\C \boe \D}(e) \otimes (c \boe d) = (T_{\C}(e) \otimes c) \boe (\unit_{\D} \otimes d) \xrightarrow{z_{e, c} \boe \hat{z}_{\unit_{\D}, d}} (c \otimes T_{\C}(e)) \boe (d \otimes \unit_{\D}) = (c \boe d) \otimes T_{\C \boe \D}(e) \]
for $e \in \E$, $c \boe d \in \C \boe \D$, $(T_{\C}(e), z_{e, -}) \in Z(\C)$ and $(\unit_{\D}, \hat{z}_{\unit_{\D}, -}) \in Z(\D)$.
Notice that $T_{\C \boe \D}(e) \simeq \unit_{\C} \boe T_{\D}(e)$. 
If $\C$ and $\D$ are fully faithful multifusion category over $\E$, $\C \boe \D$ is a fully faithful multifusion category over $\E$ by \cite[Cor.\,4.1.5]{SL}. 
If $\C$ and $\D$ are fully faithful fusion category over $\E$, $\C \boe \D$ is a fully faithful fusion category over $\E$.
\end{expl}

The 2-category $\cat$ consists of finite semisimple left $\E$-modules, left $\E$-module functors, and left $\E$-module natural transformations.
The 2-category $\cat$ with the relative tensor product $\boe$ and the unit $\E$ is a symmetric monoidal 2-category.
Let $\C, \D$ be multifusion categories. The 2-category $\BMod_{\C|\D}(\fcat)$ consists of finite semisimple $\C$-$\D$ bimodules, $\C$-$\D$ bimodule functors and $\C$-$\D$ bimodule natural transformations.

\begin{defn}[\cite{Wei}\,Def.\,4.14]
\label{defn-CD-bimodule}
 Let $\C$ and $\D$ be multifusion categories over $\E$.
The 2-category $\BMod_{\C|\D}(\cat)$ consists of the following data.
\begin{itemize}
\item An object $\M \in \BMod_{\C|\D}(\cat)$ is an object $\M$ both in $\cat$ and $\BMod_{\C|\D}(\fcat)$ equipped with monoidal natural isomorphisms $u^{\C}_e: T_{\C}(e) \odot - \simeq e \odot -$ and $u^{\D}_e: - \odot T_{\D}(e) \simeq e \odot -$ in $\Fun_{\E}(\M, \M)$ for each $e \in \E$ such that 
the functor $(c \odot - \odot d, \tilde{s}^{c \odot - \odot d})$ belongs to $\Fun_{\E}(\M, \M)$ for each $c \in \C$, $d \in \D$,
and the diagrams
\begin{equation}
\label{dig-bimodule1}
\begin{split}
 \xymatrix{
(T_{\C}(e) \otimes c) \odot - \odot d  \ar[r] \ar[d]_{z_{e,c}, 1, 1} & T_{\C}(e) \odot (c \odot - \odot d) \ar[r]^(0.55){(u^{\C}_e)_{c \odot - \odot d}} & e \odot (c \odot - \odot d) \ar[d]^{(\tilde{s}^{c \odot - \odot d}_{e,-})^{-1}} \\
(c \otimes T_{\C}(e)) \odot - \odot d \ar[r] & c \odot (T_{\C}(e) \odot -) \odot d \ar[r]^(0.55){1, u^{\C}_e, 1} & c \odot (e \odot -) \odot d
} 
\end{split}
\end{equation}
\begin{equation}
\label{dig-bimodule2} 
\begin{split}
\xymatrix{
c \odot - \odot (d \otimes T_{\D}(e)) \ar[r] \ar[d]_{1,1, \hat{z}^{-1}_{e, d}} & (c \odot - \odot d) \odot T_{\D}(e) \ar[r]^(0.55){(u^{\D}_e)_{c \odot - \odot d}} & e \odot (c \odot - \odot d) \ar[d]^{(\tilde{s}^{c \odot - \odot d}_{e,-})^{-1}} \\
c \odot - \odot (T_{\D}(e) \otimes d) \ar[r] & c \odot (- \odot T_{\D}(e)) \odot d \ar[r]^(0.55){1, u^{\D}_e, 1} & c \odot (e \odot -) \odot d
} 
\end{split}
\end{equation}
commute for all $e \in \E, c \in \C, d \in \D$, $(T_{\C}(e), z_{e,-}) \in Z(\C)$ and $(T_{\D}(e), \hat{z}_{e, -}) \in Z(\D)$.
 We use a triple $(\M, u^{\C}, u^{\D})$ to denote an object $\M$ in $\BMod_{\C|\D}(\cat)$. 

\item For objects $(\M, u^{\C}, u^{\D})$, $(\N, \bar{u}^{\C}, \bar{u}^{\D})$ in $\BMod_{\C|\D}(\cat)$, 
a 1-morphism $F: \M \rightarrow \N$ in $\BMod_{\C|\D}(\cat)$ is
a 1-morphism $F: \M \rightarrow \N$ both in $\cat$ and $\BMod_{\C|\D}(\fcat)$ such that the diagrams
\begin{equation}
\label{d2}
\begin{split}
 \vcenter{\xymatrix{
F(T_{\C}(e) \odot m) \ar[r]^(0.55){(u^{\C}_e)_m} \ar[d]_{s^F_{T_{\C}(e), m}} &  F(e \odot m)\ar[d]^{\tilde{s}^F_{e, m}} \\
T_{\C}(e) \odot F(m) \ar[r]_(0.55){(\bar{u}^{\C}_e)_{F(m)}} & e \odot F(m)
}} 
\qquad
\vcenter{\xymatrix{
F(m \odot T_{\D}(e)) \ar[r]^(0.55){(u^{\D}_e)_m} \ar[d]_{t^F_{m, T_{\D}(e)}} & F(e \odot m) \ar[d]^{\tilde{s}^F_{e, m}}  \\
F(m) \odot T_{\D}(e) \ar[r]_(0.55){(\bar{u}^{\D}_e)_{F(m)}} & e \odot F(m)
}}
\end{split}
\end{equation}
commute. Here $s^F$ (or $t^F$, $\tilde{s}^F$) is the left $\C$-module (or right $\D$-module, left $\E$-module) structure on $F$.
We use the quadruple $(F, s^F, t^F, \tilde{s}^F)$ to denote a morphism in $\BMod_{\C|\D}(\cat)$.
\item For 1-morphisms $F, G: \M \rightrightarrows \N$ in $\BMod_{\C|\D}(\cat)$, 
a 2-morphism from $F$ to $G$ is a $\C$-$\D$ bimodule natural transformation from $F$ to $G$. 
\end{itemize}
\end{defn}

 If $\D = \E$, we denote $\LMod_{\C}(\cat) \coloneqq \BMod_{\C|\E}(\cat)$. If $\C = \E$, we denote $\RMod_{\D}(\cat) \coloneqq \BMod_{\E|\D}(\cat)$.
For objects $\M, \N$ in $\BMod_{\C|\D}(\cat)$, we use $\Fun_{\C \boe \D^{\rev}}(\M, \N)$ to denote the category of 1-morphisms $\M \to \N$, 2-morphisms in $\BMod_{\C|\D}(\cat)$. We use $\Fun_{\C|\D}(\M, \N)$ to denote the category of $\C$-$\D$ bimodule functors $\M \to \N$, and $\C$-$\D$ bimodule natural transformations. 

Let $v^{\M}_e \coloneqq (u^{\D}_e)^{-1} \circ u^{\C}_e$ and $v^{\N}_e := (\bar{u}^{\D}_e)^{-1} \circ \bar{u}^{\C}_e$ for $e \in \E$, $- \in \M$. 
By (\ref{d2}), a 1-morphism $(F, s^F, t^F): \M \rightarrow \N$ in $\BMod_{\C|\D}(\cat)$ satisfies the following diagram for $e \in \E, m \in \M$:
\begin{equation}
\label{CD-bim-fun}
\begin{split}
 \xymatrix{
F(T_{\C}(e) \odot m) \ar[r]^{(v^{\M}_e)_m} \ar[d]_{s^F_{T_{\C}(e), m}} & F(m \odot T_{\D}(e)) \ar[d]^{t^F_{m, T_{\D}(e)}} \\
T_{\C}(e) \odot F(m) \ar[r]_{(v^{\N}_e)_{F(m)}} & F(m) \odot T_{\D}(e)
} 
\end{split}
\end{equation}

\begin{expl}
The category $\Fun_{\C|\D}(\M, \N)$ is a braided $\E$-module.
The left $\E$-module structure on $\Fun_{\C|\D}(\M, \N)$ is induced by $e \odot f \coloneqq T_{\C}(e) \odot f$, for $e \in \E$, $(f, s^f, t^f) \in \Fun_{\C | \D}(\M, \N)$. 
The $\E$-module braiding is given as follows:
\[ \tau^1_{e,f}: T_{\C}(e) \odot f(-) \xrightarrow{(v^{\N}_e)_{f(-)}} f(-) \odot T_{\D}(e) \]
\[ \tau^2_{f,e}: f(-) \odot T_{\D}(e) \xrightarrow{(t^f_{-, T_{\D}(e)})^{-1}} f(- \odot T_{\D}(e)) \xrightarrow{(v^{\M}_e)^{-1}} f(T_{\C}(e) \odot -) \xrightarrow{s^f_{T_{\C}(e), -}} T_{\C}(e) \odot f(-) \]
If $f \in \Fun_{\C \boe \D^{\rev}}(\M, \N)$, $\tau_{e,f} = \tau^2_{f,e} \circ \tau^1_{e, f} = \id_{e \odot f}$. 
There is an equivalence $\E \boxdot_{\E} \Fun_{\C|\D}(\M, \N) \simeq \Fun_{\C \boe \D^{\rev}}(\M, \N) $ of categories by \cite[Rem.\,3.6.17]{SL}.
\end{expl}

\begin{expl}
\label{expl-fun-3.9}
$\Fun_{\C \boe \D^{\rev}}(\M, \M)$ is a multifusion category by \cite[Cor.\,9.3.3]{EGNO}.
$\Fun_{\C \boe \D^{\rev}}(\M, \M)$ is a multifusion category over $\E$.
A functor $\hat{T}: \E \to \Fun_{\C \boe \D^{\rev}}(\M, \M)$ is defined as $e \mapsto \hat{T}^e \coloneqq T_{\C}(e) \odot -$. 
The left $\C$-module structure on $\hat{T}^e$ is defined as 
\[ (s^{\hat{T}^e})_{c, m}: T_{\C}(e) \odot (c \odot m) \simeq (T_{\C}(e) \otimes c) \odot m \xrightarrow{z_{e, c}, 1} (c \otimes T_{\C}(e)) \odot m \simeq c \odot (T_{\C}(e) \odot m)   \]
for $c \in \C, m \in \M$, and $(T_{\C}(e), z_{e,-}) \in Z(\C)$.
The right $\D$-module structure on $\hat{T}^e$ is defined as
\[ t^{\hat{T}^e}_{m, d}: T_{\C}(e) \odot (m \odot d) \simeq (T_{\C}(e) \odot m) \odot d \]
for $m \in \M$, $d \in \D$.
Check that the functor $\hat{T}^e = T_{\C}(e) \odot -$ belongs to $\Fun_{\C \boe \D^{\rev}}(\M, \M)$ for any $\tilde{e} \in \E$:
\[ \xymatrix{
T_{\C}(e) \odot (T_{\C}(\tilde{e}) \odot m) \ar[r]^{z_{e, T_{\C}(\tilde{e})}} \ar[d]_{1, (v^{\M}_{\tilde{e}})_m} & T_{\C}(\tilde{e}) \odot (T_{\C}(e) \odot m) \ar[d]^{(v^{\M}_{\tilde{e}})_{T_{\C}(e) \odot m}} \\
T_{\C}(e) \odot (m \odot T_{\D}(\tilde{e})) \ar[r] & (T_{\C}(e) \odot m) \odot T_{\D}(\tilde{e})
} \]
The above diagram commutes by the last diagram of \cite[Rem.\,4.15]{Wei} and the equation $z_{\tilde{e}, T_{\C}(e)} \circ z_{e, T_{\C}(\tilde{e})} = \id$.

The monoidal structure on $\hat{T}$ is induced by $T_{\C}(e \otimes \tilde{e}) \odot - \simeq (T_{\C}(e) \otimes T_{\C}(\tilde{e})) \odot - \simeq T_{\C}(e) \odot (T_{\C}(\tilde{e}) \odot -)$ for $e, \tilde{e} \in \E$.
The central structure of $\hat{T}$ is a natural isomorphism 
\[ \sigma_{e, g}: \hat{T}^e \circ g (m) = T_{\C}(e) \odot g(m) \xrightarrow{(s^g_{T_{\C}(e), m})^{-1}} g(T_{\C}(e) \odot m) = g \circ \hat{T}^e(m) \]
 for any $e \in \E, (g, s^g, t^g) \in \Fun_{\C \boe \D^{\rev}}(\M, \M), m \in \M$.
\end{expl}

\begin{defn}[\cite{Wei}\,Def.4.31]
\label{defn-invertible}
Let $\C, \D$ be multifusion categories over $\E$.
An $\M \in \BMod_{\C|\D}(\cat)$ is \emph{invertible} if there is an equivalence $\D^{\rev} \simeq \Fun_{\C}(\M, \M)$ of multifusion categories over $\E$.
If such an invertible $\M$ exists, $\C$ and $\D$ are said to be \emph{Morita equivalent in $\cat$}.
\end{defn}

\begin{thm}[\cite{Wei}\,Thm.\,4.40]
\label{Thm-Morita-eq}
Let $\C$ and $\D$ be fully faithful fusion categories over $\E$. $\C$ and $\D$ are Morita equivalent in $\cat$ if and only if $Z(\C, \E)$ and $Z(\D, \E)$ are equivalent as braided fusion categories over $\E$.
\end{thm}

\begin{expl}
\label{expl-fully-faith}
Suppose that $\M$ is a faithful object in $\LMod_{\C}(\cat)$ (see \cite[Def.\,4.11]{Wei}). 
There is an equivalence
\[ \Fun_{\Fun_{\C}(\M, \M)}(\M, \M) \simeq \C \]
of multifusion categories over $\E$ by \cite[Prop.4.13]{Wei}.
Then $\Fun_{\C}(\M, \M)$ is Morita equivalent to $\C^{\rev}$ in $\cat$. 
If $\E \to Z(\C)$ is fully faithful, $\E = \overline{\E} \xrightarrow{T_{\C}} \overline{Z(\C)} \cong Z(\C^{\rev})$ is fully faithful. 
We have $Z(\Fun_{\C}(\M, \M)) \simeq Z(\C^{\rev})$ by \cite[Prop.\,2.4.4]{Liang} and $\E \to Z(\Fun_{\C}(\M, \M))$ is fully faithful.
Then $Z(\Fun_{\C}(\M, \M), \E) \simeq Z(\C^{\rev}, \E)$ by \cite[Thm.\,4.34]{Wei}.
\end{expl}

\subsection{Modules over a braided fusion category over $\E$}
\label{section-3.2}

The \emph{M\"{u}ger center} of $\C$, denoted by $\C'$, is the centralizer of $\C$ in $\C$. 
Let $\M$ be a fully faithful multifusion category over $\E$. The centralizer of $\E$ in $Z(\M)$ is denoted by $Z(\M, \E)$. By Expl.\,\ref{expl-bE-mo}, $\E \boxdot_{\E} Z(\M) = \E'|_{Z(\M)} = Z(\M, \E)$.

\begin{defn}[\cite{DNO}\,Def.2.9,\,4.1]
\label{defn-braided-over-E}
 A \emph{braided fusion category over $\E$} ($\BFCe$) is a braided fusion category $\A$ equipped with a $\bk$-linear braided monoidal embedding $T_{\A}: \E \rightarrow \A'$.
A $\BFCe$ is non-degenerate if $\E = \A'$. A non-degenerate $\BFCe$ is abbreviate to a $\NBFCe$.

A \emph{braided monoidal functor over $\E$} between two $\BFCe$'s $\A$ and $\B$ is a $\bk$-linear braided monoidal functor  $F: \A \rightarrow \B$ equipped with a monoidal natural isomorphism $u_e: F(T_{\A}(e)) \simeq T_{\B}(e)$ in $\B$ for all $e \in \E$.
\end{defn}

\begin{expl}
\label{ex-center-BFCe}
$Z(\M, \E)$ is a $\NBFCe$. The braided monoidal embedding $T_{Z(\M, \E)}: \E \to Z(\M, \E)'$ is given by $e \mapsto T_{\M}(e)$.
\end{expl}

\begin{defn}[\cite{Wei}\,Def.4.41]
\label{defn-cd-braided-bim}
Let $\C, \D$ be braided fusion categories over $\E$. 
The 2-category $\BMod_{\C|\D}(\Alg(\cat))$ consists of the following data.
\begin{itemize}
\item An object $\M \in \BMod_{\C|\D}(\Alg(\cat))$ is a fully faithful multifusion category $\M$ over $\E$ equipped with a $\bk$-linear braided monoidal functor $\phi_{\M}: \overline{\C} \boe \D \to Z(\M, \E)$ over $\E$.  
An object $\M \in \BMod_{\C|\D}(\Alg(\cat))$ is closed if $\phi_{\M}$ is an equivalence.
\item For objects $\M, \N$ in $\BMod_{\C|\D}(\Alg(\cat))$, a 1-morphism $F: \M \rightarrow \N$ in $\BMod_{\C|\D}(\Alg(\cat))$ is a $\bk$-linear monoidal functor $F: \M \rightarrow \N$ equipped with a monoidal natural isomorphism $u^{\M\N}: F \circ \phi_{\M} \Rightarrow \phi_{\N}$ such that the diagram
\begin{equation}
\label{braided-functor-E}
\begin{split}
 \xymatrix{
F(\phi_{\M}(x) \otimes m) \ar[r]^(0.45){J^F_{\phi_{\M}(x), m}} \ar[d]_{\beta^{\M}_{x,m}} & F(\phi_{\M}(x)) \otimes F(m) \ar[r]^{u^{\M\N}_x, 1} & \phi_{\N}(x) \otimes F(m) \ar[d]^{\beta^{\N}_{x,F(m)}} \\
F(m \otimes \phi_{\M}(x)) \ar[r]_(0.45){J^F_{m, \phi_{\M}(x)}} & F(m) \otimes F(\phi_{\M}(x)) \ar[r]_{1, u^{\M\N}_x} & F(m) \otimes \phi_{\N}(x)
} 
\end{split}
\end{equation}
commutes for $x \in \overline{\C} \boe \D, m \in \M$, where $(\phi_{\M}(x), \beta^{\M}_{x,-})  \in Z(\M, \E)$ and $(\phi_{\N}(x), \beta^{\N}_{x, -}) \in Z(\N, \E)$ and $J^F$ is the monoidal structure of $F$.

\item For 1-morphisms $F, G: \M \rightrightarrows \N$ in $\BMod_{\C|\D}(\Alg(\cat))$, a 2-morphism $\alpha: F \Rightarrow G$ in $\BMod_{\C|\D}(\Alg(\cat))$ is a monoidal isomorphism $\alpha$ such that the diagram 
\[ 
\xymatrixcolsep{0.5pc}
\xymatrixrowsep{1.2pc}
\xymatrix{
F(\phi_{\M}(x)) \ar[rr]^{\alpha_{\phi_{\M}(x)}} \ar[rd]_{u^{\M\N}_x} && G(\phi_{\M}(x)) \ar[ld]^{\tilde{u}^{\M\N}_x} \\
&\phi_{\N}(x) &
} \]
commutes for $x \in \overline{\C} \boe \D$, where $u^{\M\N}$ and $\tilde{u}^{\M\N}$ are the monoidal isomorphisms on $F$ and $G$ respectively.
\end{itemize}
\end{defn}

\begin{prop}
For $(\M, \phi_{\M})$, $(\N, \phi_{\N}) \in \BMod_{\C|\D}(\Alg(\cat))$, $F: \M \to \N$ is a 1-morphism in $\BMod_{\C|\D}(\Alg(\cat))$. 
Then $F$ is a monoidal functor over $\E$ and if $\phi_{\M}(x) \in Z(\M, \E)$ for $x \in \overline{\C} \boe \D$, $\phi_{\N}(x) \in Z(\N, \E)$.
\end{prop}
\begin{proof}
It is sufficient to prove the statement in the case $\overline{\C} = \E$.
Recall the central structure on the functor $T_{Z(\M, \E)} = T_{\M}: \E \to Z(\M, \E)'$ by Expl.\,\ref{ex-center-BFCe}.
Let $u^{\M}: \phi_{\M} \circ T_{\D} \Rightarrow T_{\M}$ and $u^{\N}: \phi_{\N} \circ T_{\D} \Rightarrow T_{\N}$ be the structures of monoidal functors over $\E$ on $\phi_{\M}$ and $\phi_{\N}$ respectively.
The structure of monoidal functor over $\E$ on $F$ is induced by the composition $v: F \circ T_{\M} \xRightarrow{1, (u^{\M})^{-1}} F \circ \phi_{\M} \circ T_{\D} \xRightarrow{u^{\M\N}, 1} \phi_{\N} \circ T_{\D} \xRightarrow{u^{\N}} T_{\N}$.
Next check that $F: \M \to \N$ is a monoidal functor over $\E$.
\[ \xymatrix@=5ex{
F(T_{\M}(e) \otimes m) \ar[r]^{J^F} \ar[d]^{(u^{\M}_e)^{-1},1} \ar@/_4.3pc/[ddd]_{\sigma_{e, m}} & F(T_{\M}(e)) \otimes F(m) \ar[r]^{v_e,1} \ar[d]^{(u^{\M}_e)^{-1},1} & T_{\N}(e) \otimes F(m) \ar@/^4.3pc/[ddd]^{\hat{\sigma}_{e, F(m)}} \\
F(\phi_{\M}(T_{\D}(e)) \otimes m) \ar[r]^{J^F} \ar[d]^{\beta^{\M}_{T_{\D}(e), m}} & F(\phi_{\M}(T_{\D}(e))) \otimes F(m) \ar[r]^{u^{\M\N}_{T_{\D}(e)},1}  & \phi_{\N}(T_{\D}(e)) \otimes F(m) \ar[u]^{u^{\N}_e,1} \ar[d]_{\beta^{\N}_{T_{\D}(e), F(m)}} \\
F(m \otimes \phi_{\M}(T_{\D}(e))) \ar[r]^{J^F} \ar[d]^{1, u^{\M}_e} & F(m) \otimes F(\phi_{\M}(T_{\D}(e))) \ar[r]^{1,u^{\M\N}_{T_{\D}(e)}} \ar[d]^{1, u^{\M}_e} & F(m) \otimes \phi_{\N}(T_{\D}(e)) \ar[d]_{1, u^{\N}_e} \\
F(m \otimes T_{\M}(e)) \ar[r]^{J^F} & F(m) \otimes F(T_{\M}(e)) \ar[r]^{1, v_e} & F(m) \otimes T_{\N}(e)
} \]
The middle hexagon commutes by the diagram (\ref{braided-functor-E}).
The left-upper and left-down squares commute by the naturality of $J^F$.
The right-upper and right-down squares commute by the definition of $v_e$.
Since $u^{\M}_e: \phi_{\M}(T_{\D}(e)) \simeq T_{\M}(e)$ is the isomorphism in $Z(\M, \E)$, the left-most square commutes for $(T_{\M}(e), \sigma_{e,-}), (\phi_{\M}(T_{\D}(e)), \beta^{\M}_{T_{\D}(e), -}) \in Z(\M, \E)$.
Since $u^{\N}_e: \phi_{\N}(T_{\D}(e)) \simeq T_{\N}(e)$ is the isomorphism in $Z(\N, \E)$, the right-most square commutes for $(\phi_{\N}(T_{\D}(e)), \beta^{\N}_{T_{\D}(e), -}), (T_{\N}(e), \hat{\sigma}_{e,-}) \in Z(\N, \E)$. 
Consider the following diagram:
\[\xymatrix@=4ex{
F(\phi_{\M}(x) \otimes T_{\M}(e)) \ar[r] \ar[d]_{\beta^{\M}_{x,T_{\M}(e)}} & F(\phi_{\M}(x)) \otimes F(T_{\M}(e)) \ar[r]^{u^{\M\N}_x, 1} & \phi_{\N}(x) \otimes F(T_{\M}(e)) \ar[d]^{\beta^{\N}_{x,F(T_{\M}(e))}} \ar[r]^{1, v_e} & \phi_{\N}(x) \otimes T_{\N}(e) \ar[d]^{\beta^{\N}_{x, T_{\N}(e)}} \\
F(T_{\M}(e) \otimes \phi_{\M}(x)) \ar[r] \ar[dd]_{\sigma_{e, \phi_{\M}(x)}} & F(T_{\M}(e)) \otimes F(\phi_{\M}(x)) \ar[r]^{1, u^{\M\N}_x} \ar[dr]_{v_e, 1} & F(T_{\M}(e)) \otimes \phi_{\N}(x)  \ar[r]^{v_e, 1}   & T_{\N}(e) \otimes \phi_{\N}(x) \ar[dd]^{\hat{\sigma}_{e, \phi_{\N}(x)}} \\
& & T_{\N}(e) \otimes F(\phi_{\M}(x)) \ar[ru]_{1, u^{\M\N}_x} \ar[d]_{\hat{\sigma}_{e, F(\phi_{\M}(x))}}  & \\
F(\phi_{\M}(x) \otimes T_{\M}(e)) \ar[r]  & F(\phi_{\M}(x)) \otimes F(T_{\M}(e))  \ar[r]_{1, v_e}&   F(\phi_{\M}(x)) \otimes T_{\N}(e) \ar[r]_{u^{\M\N}_x, 1} &\phi_{\N}(x) \otimes T_{\N}(e)
} \]
The upper hexagon commutes by the diagram (\ref{braided-functor-E}).
The bottom hexagon commutes since $F: \M \to \N$ is a monoidal functor over $\E$.
The right-upper and right-bottom squares commute by the naturalities of $\beta^{\N}_{x,-}$ and $\hat{\sigma}_{e, -}$. 
Then if $\sigma_{e, \phi_{\M}(x)} \circ \beta^{\M}_{x, T_{\M}(e)} = \id$, $\hat{\sigma}_{e, \phi_{\N}(x)} \circ \beta^{\N}_{x, T_{\N}(e)} = \id$.
\end{proof}

Let $\C, \D$ be fully faithful multifusion category over $\E$ and $\M, \N \in \BMod_{\C|\D}(\cat)$.
$\Fun_{\C|\D}(\M, \M)$ is a monoidal $Z(\C)$-$Z(\D)$ bimodule by \cite[Rem.\,2.6.3]{Liang}. That is, a finite monoidal category $\Fun_{\C|\D}(\M, \M)$ equipped with a $\bk$-linear braided monoidal functor $\phi: \overline{Z(\C)} \boxtimes Z(\D) \to Z(\Fun_{\C|\D}(\M, \M))$.
By the universal property of $\boxdot_{\E}$, we get
 \[ \xymatrix{
 \E \boxdot_{\E} (\overline{Z(\C)} \boxtimes Z(\D)) \ar[r]^{1, \boxdot_{\E}} \ar[d]_{1, \phi} &  \E \boxdot_{\E} (\overline{Z(\C)} \boxdot_{\E} Z(\D)) \ar[ld]^{\exists ! \underline{\phi}} \\
 \E \boxdot_{\E} Z(\Fun_{\C|\D}(\M, \M)) &
 } \]
$\underline{\phi}$ induces a functor $\overline{Z(\C, \E)} \boe Z(\D, \E) \simeq (\E \boxdot_{\E} \overline{Z(\C)}) \boe (\E \boxdot_{\E} Z(\D)) \simeq \E \boxdot_{\E} (\overline{Z(\C)} \boxdot_{\E} Z(\D)) \xrightarrow{\underline{\phi}} \E \boxdot_{\E} Z(\Fun_{\C|\D}(\M, \M)) \simeq Z(\Fun_{\C | \D}(\M, \M), \E)$.
The first equivalence is due to $\E \boxdot_{\E} \overline{Z(\C)} \simeq \E \boxdot_{\E} Z(\C^{\rev}) = Z(\C^{\rev}, \E) \simeq \overline{Z(\C, \E)}$.
The second equivalence is due to Rem.\,\ref{rem-sun-prop-bE}.
Then $\Fun_{\C|\D}(\M, \M)$ is equipped with a $Z(\C, \E)$-$Z(\D, \E)$ bimodule.

\begin{expl}
\label{expl-braided-bimodule}
Let $\C$ and $\D$ be fully faithful multifusion categories over $\E$ and $\M \in \BMod_{\C|\D}(\cat)$. Then $\Fun_{\C \boe \D^{\rev}}(\M, \M) \in \BMod_{Z(\C, \E)| Z(\D, \E)}(\Alg(\cat))$.
Since $\E \to Z(\C)$ and $\E \to Z(\D)$ are fully faithful, $\E \to Z(\Fun_{\C \boe \D^{\rev}}(\M, \M))$ is fully faithful by Expl.\,\ref{ex3} and Expl.\,\ref{expl-fully-faith}.
A braided monoidal functor 
\[ \phi: \overline{Z(\C, \E)} \boe Z(\D, \E) \to Z(\Fun_{\C \boe \D^{\rev}}(\M, \M), \E) \]
over $\E$ is defined as an object $\phi(c \boe d) = c \odot - \odot d$ in $\Fun_{\C \boe \D^{\rev}}(\M, \M)$, equipped with a half-braiding 
$c \odot f(-) \odot d \xrightarrow{(s^f_{c,-})^{-1}} f(c \odot -) \odot d \xrightarrow{(t^f_{c \odot -, d})^{-1}} f(c \odot - \odot d)$
for $(c, \beta_{c,-}) \in \overline{Z(\C, \E)}, (d, \gamma_{d,-}) \in Z(\D, \E)$ and $(F, s^f, t^f) \in \Fun_{\C \boe \D^{\rev}}(\M, \M)$.
The left $\C$-module structure $s^{c \odot - \odot d}$ and right $\D$-module structure $t^{c \odot - \odot d}$ on $c \odot - \odot d$ are defined as 
\[  s^{c \odot - \odot d}_{c_1, -}: (c \odot (c_1 \odot -)) \odot d \xrightarrow{s^{c \odot -}_{c_1, -} = \beta_{c, c_1},1} (c_1 \odot (c \odot -)) \odot d  \xrightarrow{s^{- \odot d}_{c_1, c \odot -}} c_1 \odot ((c \odot -) \odot d)\]
\[ t^{c \odot - \odot d}_{-, d_1}: (c \odot (- \odot d_1)) \odot d \xrightarrow{t^{c \odot -}_{-, d_1}, 1} ((c \odot -) \odot d_1) \odot d \xrightarrow{t^{- \odot d}_{c \odot -, d_1} = \gamma_{d, d_1}^{-1}}  ((c \odot -) \odot d) \odot d_1\]
for $c_1 \in \C, d_1 \in \D$.
Check that $c \odot - \odot d \in \Fun_{\C \boe \D^{\rev}}(\M, \M)$ (see diagram (\ref{CD-bim-fun})).
The functor $(c \odot -, s^{c \odot -}, t^{c \odot -})$ belongs to $\Fun_{\C \boe \D^{\rev}}(\M, \M)$ if the right square of the following diagram
\[ \xymatrix{
(c \otimes T_{\C}(e)) \odot m \ar[r]  \ar[d]^{\beta_{c, T_{\C}(e)}} & c \odot (T_{\C}(e) \odot m) \ar[r]^{1, (v^{\M}_e)_m}  \ar@{.>}[d]^{s^{c \odot -}_{T_{\C}(e), m}} & c \odot (m \odot T_{\D}(e)) \ar[d]^{t^{c \odot -}_{m, T_{\D}(e)}} \\
(T_{\C}(e) \otimes c) \odot m \ar[r]  \ar@/^1pc/[u]^{z_{e, c}} & T_{\C}(e) \odot (c \odot m) \ar[r]_{(v^{\M}_e)_{c \odot m}}  & (c \odot m) \odot T_{\D}(e)
} \]
commutes for $m \in \M$ and $(T_{\C}(e), z_{e, -}) \in Z(\C)$.
The left square commutes by the definition of $s^{c \odot -}$ and the equation $\beta_{c, T_{\C}(e)} \circ z_{e, c} = \id$. The outward diagram commutes by \cite[Rem.\,4.15]{Wei}. 
The functor $(- \odot d, s^{- \odot d}, t^{- \odot d})$ belongs to $\Fun_{\C \boe \D^{\rev}}(\M, \M)$ if the right square of the following diagram
\[ \xymatrix{
m \odot (T_{\D}(e) \otimes d) \ar[r]   & (m \odot T_{\D}(e)) \odot d \ar[r]^{(v^{\M}_e)_m^{-1}, 1} \ar@{.>}[d]^{t^{- \odot d}_{m, T_{\D}(e)}} & (T_{\C}(e) \odot m) \odot d  \ar[d]^{s^{- \odot d}_{T_{\C}(e), m}} \\
m \odot (d \otimes T_{\D}(e)) \ar[r] \ar@/^1pc/[u]^{1, \hat{z}^{-1}_{e, d}} \ar[u]_{1, \gamma_{d, T_{\D}(e)}}& (m \odot d) \odot T_{\D}(e) \ar[r]^{(v^{\M}_e)_{m \odot d}^{-1}} & T_{\C}(e) \odot (m \odot d)
} \]
commutes for $m \in \M$ and $(T_{\D}(e), \hat{z}_{e,-}) \in Z(\D)$.
The left square commutes by the definition of $t^{- \odot d}$ and the equation $\gamma_{d, T_{\D}(e)} \circ \hat{z}_{e, d} = \id$.
The outward diagram commutes by \cite[Rem.\,4.15]{Wei}.
The functor $(c \odot - \odot d, s^{c \odot - \odot d}, t^{c \odot - \odot d})$ belongs to $\Fun_{\C \boe \D^{\rev}}(\M, \M)$ by the commutative outward diagram
\[ \xymatrix{
(c \odot (T_{\C}(e) \odot m)) \odot d  \ar[d]_{1, (v^{\M}_e)_m, 1} \ar[r]^{s^{c \odot -}_{T_{\C}(e), m}, 1} & (T_{\C}(e) \odot (c \odot m)) \odot d \ar[d]^{(v^{\M}_e)_{c \odot m}, 1} \ar[r]^{s^{- \odot d}_{T_{\C}(e), c \odot m}} & T_{\C}(e) \odot ((c \odot m) \odot d) \ar[d]^{(v^{\M}_e)_{(c \odot m) \odot d}}  \\
(c \odot (m \odot T_{\D}(e))) \odot d \ar[r]_{t^{c \odot -}_{m, T_{\D}(e)}, 1} & ((c \odot m) \odot T_{\D}(e)) \odot d \ar[r]_{t^{- \odot d}_{c \odot m, T_{\D}(e)}} & ((c \odot m) \odot d) \odot T_{\D}(e) 
} \]
The left square commutes since $(c \odot -, s^{c \odot -}, t^{c \odot -})$ belongs to $\Fun_{\C \boe \D^{\rev}}(\M, \M)$. The right square commutes since $(- \odot d, s^{- \odot d}, t^{- \odot d})$ belongs to $\Fun_{\C \boe \D^{\rev}}(\M, \M)$.

Recall Expl.\,\ref{expl-fun-3.9}.
An object $(f, s^f, t^f) \in Z(\Fun_{\C \boe \D^{\rev}}(\M, \M), \E)$ is an object $(f, \sigma_{f, -}: f \circ - \to - \circ f) \in Z(\Fun_{\C \boe \D^{\rev}}(\M, \M))$ such that 
$\big( \hat{T}^e \circ f \xrightarrow{(s^f)^{-1}} f \circ \hat{T}^e \xrightarrow{\sigma_{f, \hat{T}^e}} \hat{T}^e \circ f \big) = \id_{\hat{T}^e \circ f}$.
The functor $c \odot - \odot d$ belongs to $Z(\Fun_{\C \boe \D^{\rev}}(\M, \M), \E)$ if the composition
\[ T_{\C}(e) \odot ((c \odot -) \odot d) \xrightarrow{(s^{- \odot d}_{T_{\C}(e), c \odot -})^{-1}} (T_{\C}(e) \odot (c \odot -)) \odot d \xrightarrow{(s^{c \odot -}_{T_{\C}(e), -})^{-1} = \beta^{-1}_{c, T_{\C}(e)}} (c \odot (T_{\C}(e) \odot -)) \odot d \]\[\xrightarrow{(s^{\hat{T}^e}_{c, -})^{-1} =z^{-1}_{e, c}} (T_{\C}(e) \odot (c \odot -)) \odot d \xrightarrow{(t^{\hat{T}^e}_{c \odot -, d})^{-1}} T_{\C}(e) \odot ((c \odot -) \odot d) \]
equals to identity. Since $\beta_{c, T_{\C}(e)} \circ z_{e, c} = \id$, the above composition equals to identity.

The monoidal structure on $\phi$ is defined as
$(c_1 \otimes c_2) \odot - \odot (d_1 \otimes d_2) \xrightarrow{1,1,1, \gamma_{d_1, d_2}} c_1 \odot (c_2 \odot - \odot d_2) \odot d_1$
for $(c_1, \beta_{c_1, -}), (c_2, \beta_{c_2, -}) \in \overline{Z(\C, \E)}$, $(d_1, \gamma_{d_1, -}), (d_2, \gamma_{d_2, -}) \in Z(\D, \E)$.
Recall Expl.\,\ref{ex3}, \ref{ex-center-BFCe} and \ref{expl-fun-3.9}.
The braided monoidal functor $\phi$ equipped with the monoidal natural isomorphism $\phi(T_{\C}(e) \boe \unit_{\D}) = T_{\C}(e) \odot - \odot \unit_{\D} \simeq T_{\C}(e) \odot - = \hat{T}^e$ is the braided monoidal functor over $\E$.
\end{expl}

\subsection{Functoriality of the center}

\begin{prop}
\label{prop-C-3.16}
Let $\C$ be a multifusion category over $\E$. There are equivalences of multifusion categories over $\E$:
\[ \Fun_{\C^{\rev}}(\C, \C) \simeq \C \qquad \Fun_{\C}(\C, \C) \simeq \C^{\rev} \]
\end{prop}
\begin{proof}
We have a functor $\Phi: \Fun_{|\C}(\C, \C) \to \C$ defined by
$F \mapsto F(\unit_{\C}), (F \xrightarrow{\phi} G) \mapsto (F(\unit_{\C}) \xrightarrow{\phi_{\unit_{\C}}} G(\unit_{\C}))$.
The inverse of $\Phi$ is defined as $\Psi: \C \to \Fun_{|\C}(\C, \C)$, $x \mapsto x \otimes -$. 
The central structure on the monoidal functor $\E \to \Fun_{|\C}(\C, \C)$, $e \mapsto - \otimes T_{\C}(e)$ is a natural isomorphism 
$\sigma_{e, g}: g(-) \otimes T_{\C}(e) \simeq g(- \otimes T_{\C}(e))$
for any $e \in \E$, $g \in \Fun_{|\C}(\C, \C)$.
The monoidal functor $\Psi$ equipped with a monoidal natural isomorphism $u_e: T_{\C}(e) \otimes - \xrightarrow{c_{e, -}} - \otimes T_{\C}(e)$ is a monoidal functor over $\E$, where $(T_{\C}(e), c_{e,-}) \in Z(\C)$. 
\end{proof}

\begin{expl}
\label{expl-fun}
Let $\C$ be a multifusion category over $\E$ and $(\M, u^{\M}), (\N, u^{\N}) \in \LMod_{\C}(\cat)$. Then $\Fun_{\E}(\M, \N) \in \BMod_{\C|\C}(\cat)$.
The $\C$-$\C$ bimodule structure on $\Fun_{\E}(\M, \N)$ is defined as $c \odot g \odot c' \coloneqq c \odot g( c' \odot -)$ for $e \in \E, c, c' \in \C, g \in \Fun_{\E}(\M, \N)$.
$\Fun_{\E}(\M, \N)$ equipped with the monoidal natural isomorphism $v^{\Fun}_e$  belongs to $\BMod_{\C|\C}(\cat)$, where
\[ (v^{\Fun}_e)_g: (T_{\C}(e) \odot g)(-) = T_{\C}(e) \odot g(-) \xrightarrow{(u^{\N}_e)_{g(-)}}  e \odot g(-) \simeq  g(e \odot -) \xrightarrow{(u^{\M}_e)^{-1}} g(T_{\C}(e) \odot -) = (g \odot T_{\C}(e))(-). \]
\end{expl}

\begin{prop}
\label{prop-3.18}
Let $\C$ and $\D$ be multifusion categories over $\E$ and $(\M, u^{\M}, \bar{u}^{\M}), (\M', u^{\M'}, \bar{u}^{\M'})$ belong to $\BMod_{\C|\D}(\cat)$. There is an equivalence
\[ \Fun_{\C \boe \C^{\rev}} (\C, \Fun_{\D^{\rev}}(\M, \M')) \simeq \Fun_{\C \boe \D^{\rev}}(\M, \M') \]
If $\D = \E$, we have $\Fun_{\C \boe \C^{\rev}}(\C, \Fun_{\E}(\M, \M')) \simeq \Fun_{\C}(\M, \M')$.
\end{prop}
\begin{proof}
The category $\C$ equipped with the monoidal natural isomorphism $T_{\C}(e) \otimes - \xrightarrow{z_{e,-}} - \otimes T_{\C}(e)$ belongs to $\BMod_{\C|\C}(\cat)$ for $e \in \E$ and $(T_{\C}(e), z_{e,-}) \in Z(\C)$.  
We define a functor
\[ \Phi: \Fun_{\C \boe \C^{\rev}}(\C, \Fun_{\E \boe \D^{\rev}}(\M, \M')) \to \Fun_{\C \boe \D^{\rev}}(\M, \M'), \qquad f \mapsto f^{\unit_{\C}} \]
 Since $(f, s^f, t^f, \tilde{s}^f): \C \to \Fun_{\E \boe \D^{\rev}}(\M, \M')$, $c \mapsto (f^c, \tilde{s}^{f^c}, t^{f^c})$, then $f$ and $f^c$ satisfy the diagrams (see diagram (\ref{CD-bim-fun}) and the right diagram of (\ref{d2})) 
 \begin{equation}
 \label{diag-3-first}
 \begin{split}
  \vcenter{\xymatrix{
 f^{T_{\C}(e) \otimes c} \ar[r]^{z_{e, c}} \ar[d]_{s^f_{T_{\C}(e), c}} & f^{c \otimes T_{\C}(e)} \ar[d]^{t^f_{c, T_{\C}(e)}} \\
 T_{\C}(e) \odot f^c \ar[r]_{(v^{\Fun}_e)_{f^c}} & f^c \odot T_{\C}(e)
 }}
 \qquad
 \vcenter{\xymatrix{
 f^c(e \odot m) \ar[r]^(0.4){(\bar{u}^{\M}_e)^{-1}_{m}} \ar[d]_{\tilde{s}^{f^c}_{e, m}} & f^c(m \odot T_{\D}(e)) \ar[d]^{t^{f^c}_{m, T_{\D}(e)}} \\
 e \odot f^c(m) \ar[r]_(0.45){(\bar{u}^{\M'}_e)^{-1}_{f^c(m)}} & f^c(m) \odot T_{\D}(e)
 }} 
 \end{split}
 \end{equation}
 for any $c \in \C$, $e \in \E$, $m \in \M$.
 The left $\C$-module structure on the functor $f^{\unit_{\C}}$ is induced by
 \[ s^{f^{\unit_{\C}}}_{c, m}: f^{\unit_{\C}}(c \odot m) = (f^{\unit_{\C}} \odot c)(m) \xrightarrow{(t^f_{\unit_{\C}, c})^{-1}} f^{\unit_{\C} \otimes c}(m) \xrightarrow{z_{\unit_{\C}, c}} f^{c \otimes \unit_{\C}}(m) \xrightarrow{s^f_{c, \unit_{\C}}} (c \odot f^{\unit_{\C}})(m) = c \odot f^{\unit_{\C}}(m) \]
 for $c \in \C$ and $(\unit_{\C}, z_{\unit_{\C},-}) \in Z(\C)$.
Next check that $(f^{\unit_{\C}}, s^{f^{\unit_{\C}}}, t^{f^{\unit_{\C}}}, \tilde{s}^{f^{\unit_{\C}}})$ satisfies the right two squares below.
\[ \xymatrix{
f^{\unit_{\C} \otimes T_{\C}(e)}(m) \ar[r]^(0.3){t^f_{\unit_{\C}, T_{\C}(e)}} \ar[d]_{z^{-1}_{e, \unit_{\C}}} \ar@/^1pc/[d]^{z_{\unit_{\C}, T_{\C}(e)}} & (f^{\unit_{\C}} \odot T_{\C}(e)) (m) = f^{\unit_{\C}}(T_{\C}(e) \odot m) \ar[r]^(0.7){(u^{\M}_e)_m} \ar@<+13mm>[d]_{s^{f^{\unit_{\C}}}_{T_{\C}(e), m}} & f^{\unit_{\C}}(e \odot m) \ar[d]_{\tilde{s}^{f^{\unit_{\C}}}_{e,m}} \ar[r]^(0.45){(\bar{u}^{\M}_e)^{-1}_m} & f^{\unit_{\C}}(m \odot T_{\D}(e)) \ar[d]^{t^{f^{\unit_{\C}}}_{m, T_{\D}(e)}} \\
f^{T_{\C}(e) \otimes \unit_{\C}}(m) \ar[r]_(0.3){s^f_{T_{\C}(e), \unit_{\C}}} & (T_{\C}(e) \odot f^{\unit_{\C}})(m) = T_{\C}(e) \odot f^{\unit_{\C}}(m) \ar[r]_(0.7){(u^{\M'}_e)_{f^{\unit_{\C}}(m)}} \ar@<14mm>@{.>}[u]_{(v^{\Fun}_e)_{f^{\unit_{\C}}}}  &  e \odot f^{\unit_{\C}}(m) \ar[r]_{(\bar{u}^{\M'}_e)^{-1}_{f^{\unit_{\C}}(m)}} & f^{\unit_{\C}}(m) \odot T_{\D}(e)
} \]
The left-most square commutes by the left square of (\ref{diag-3-first}).
By \cite[Prop.\,A.3]{Wei}, the equation $z_{e, \unit_{\C}} \circ z_{\unit_{\C}, T_{\C}(e)} = \id$ holds. 
 By the definitions of $(v^{\Fun}_e)_{f^{\unit_{\C}}}$ and $s^{f^{\unit_{\C}}}$, the middle square commutes.
The right-most square commutes by the right diagram of (\ref{diag-3-first}).
Then $f^{\unit_{\C}}$ belongs to $\Fun_{\C \boe \D^{\rev}}(\M, \M')$.

Conversely,  define a functor 
\[ \Psi: \Fun_{\C \boe \D^{\rev}}(\M, \M') \to \Fun_{\C \boe \C^{\rev}}(\C, \Fun_{\E \boe \D^{\rev}}(\M, \M')), \qquad g \mapsto (g: c \mapsto g^c ) \]
where $g^c \coloneqq g(c \odot -)$.
Given a functor $(g, s^g, t^g, \tilde{s}^g) \in \Fun_{\C \boe \D^{\rev}}(\M, \M')$, the $\C$-bimodule structure on $g: \C \to \Fun_{\E \boe \D^{\rev}}(\M, \M')$ is defined as
\[ g^{\tilde{c} \otimes c} = g((\tilde{c} \otimes c) \odot -) \simeq g(\tilde{c} \odot (c \odot -)) \xrightarrow{s^g_{\tilde{c}, c \odot -}} \tilde{c} \odot g(c \odot -) = \tilde{c} \odot g^c \]
\[ g^{c \otimes \tilde{c}} = g((c \otimes \tilde{c}) \odot -) \simeq g(c \odot (\tilde{c} \odot -)) = g^c(\tilde{c} \odot -) = g^c \odot \tilde{c} \]
for $\tilde{c} \in \C$. The left $\E$-module structure $\tilde{s}^{g^c}$ and right $\D$-module structure $t^{g^c}$ on $g^c$ are defined as
\[ \tilde{s}^{g^c}_{e,-}: g^c(e \odot -) =  g(c \odot (e \odot -))  \xrightarrow{s^{c \odot -}_{e, -}}  g(e \odot (c \odot -)) \xrightarrow{\tilde{s}^g_{e, c \odot -}}  e \odot g(c \odot -) = e \odot g^c(-) \]
\[ t^{g^c}_{-, d}:  g^c(- \odot d) = g(c \odot (- \odot d))  \xrightarrow{\simeq} g((c \odot -) \odot d) \xrightarrow{t^g_{c \odot-, d}}  g(c \odot -) \odot d = g^c(-) \odot d \]
for $e \in \E$, $(c \odot -, s^{c \odot -}) \in \Fun_{\E}(\M, \M)$.
By Expl.\,\ref{expl-fun}, $(v^{\Fun}_e)_{g^c}$ is defined as
\[ T_{\C}(e) \odot g^c = T_{\C}(e) \odot g^c(-) \xrightarrow{(u^{\M'}_e)_{g^c(-)}} e \odot g^c(-) \xrightarrow{(\tilde{s}^{g^c}_{e,-})^{-1}} g^c(e \odot -) \xrightarrow{(u^{\M}_e)^{-1}}  g^c(T_{\C}(e) \odot -) = g^c \odot T_{\C}(e) \]
Check that $(g^c, \tilde{s}^{g^c}, t^{g^c})$ belongs to $\Fun_{\E \boe \D^{\rev}}(\M, \M')$ (see the right diagram of (\ref{d2})).
\[ \xymatrix{
g(c \odot (e \odot -)) \ar[d]_{(\bar{u}^{\M}_e)^{-1}} \ar[r]^{s^{c \odot -}_{e,-}}& g(e \odot (c \odot -)) \ar[d]^{(\bar{u}^{\M}_e)^{-1}_{c \odot -}} \ar[r]^{\tilde{s}^g_{e, c \odot -}} & e \odot g(c \odot -)  \ar[d]^{(\bar{u}^{\M'}_e)^{-1}_{g(c \odot -)}} \\
g(c \odot (- \odot T_{\D}(e))) \ar[r]_{\simeq} & g((c \odot -) \odot T_{\D}(e)) \ar[r]_{t^g_{c \odot-, T_{\D}(e)}} & g(c \odot -) \odot T_{\D}(e)
} \]  
The left square commutes by \cite[Rem.\,4.15]{Wei}.
The right square commutes since $g \in \Fun_{\E \boe \D^{\rev}}(\M, \M')$ (see the right diagram of (\ref{d2})). 
Then check $g$ belongs to $\Fun_{\C \boe \C^{\rev}}(\C, \Fun_{\E \boe \D^{\rev}}(\M, \M'))$.
\[ \xymatrix{
g^{T_{\C}(e) \otimes c} = g((T_{\C}(e) \otimes c) \odot -) \ar[dd]_{z_{e, c}, 1} \ar[r]^{s^g_{T_{\C}(e), c \odot -}} \ar[rrd]_{(u^{\M}_e)_{c \odot -}} & T_{\C}(e) \odot g(c \odot -) = T_{\C}(e) \odot g^c \ar[r]^(0.67){(u^{\M'}_e)_{g(c \odot -)}} \ar@<+12mm>@{.>}[dd]^(0.7){(v^{\Fun}_e)_{g^c}} & e \odot g(c \odot -) \ar[d]^{(\tilde{s}^g_{e, c \odot -})^{-1}} \\
& & g(e \odot (c \odot -)) \ar[d]^{(\tilde{s}^{c \odot -}_{e, -})^{-1}} \\
g^{c \otimes T_{\C}(e)} = g((c \otimes T_{\C}(e)) \odot -) \ar[r] & g(c \odot T_{\C}(e) \odot -) = g^c \odot T_{\C}(e) & g(c \odot e \odot -)  \ar[l]_(0.3){1, (u^{\M}_e)^{-1}}
} \]
By the definition of $(v^{\Fun}_e)_{g^c}$, we want to check the outward diagram commutes.  The upper square commutes since $g \in \Fun_{\C \boe \E}(\M, \M')$ (see the left diagram of (\ref{d2})). The lower pentagon commutes since $(\M, u^{\M}) \in \LMod_{\C}(\cat)$ (see the diagram (\ref{dig-bimodule1})).
Check that $\Psi \circ \Phi \simeq \id$ and $\Phi \circ \Psi \simeq \id$.
\end{proof} 

Let $\C$ and $\D$ be rigid monoidal categories and $\CP$ a $\C$-$\D$ bimodule. The $\D$-$\C$ bimodule structures on $\CP^{R|\op|L}$ is defined as $d \odot^R p \odot^L c \coloneqq c^L \odot p \odot d^R$ for $c \in \C$, $d \in \D$ and $p \in \CP$.
For a multifusion category $\C$ over $\E$, we denote $I_{\C} = [\unit_{\C}, \unit_{\C}]^R_{\C \boe \C^{\rev}} \in \C \boe \C^{\rev}$.
Given a left $\C$-module $\M$, we use ${}^{LL}_{\C}\!\M$ to denote the left $\C$-module which has the same underlying category as $\M$ but equipped with the action $a \odot^{LL} m = a^{LL} \odot m$ for $a \in \C, m \in \M$.
Given a right $\C$-module $N$, we use $\N^{RR}_{\C}$ to denote ${}^{LL}_{\C^{\rev}}\!\N$.

\begin{lem}
\label{lem-1.30}
Let $\C$ be a multifusion category over $\E$ and $\M, \M' \in \LMod_{\C}(\cat)$. There is an equivalence 
\[ \C^{\op|L} \boxtimes_{\C^{\rev} \boe \C} \Fun_{\E}(\M, \M') \simeq \Fun_{\C}(\M, \M') \]
which maps $\unit_{\C} \boxtimes_{\C^{\rev} \boe \C} f \mapsto I_{\C} \odot f$, and an equivalence
\[ \C \boxtimes_{\C^{\rev} \boe \C} \Fun_{\E}({}^{LL}_{\C}\!\M, \M') \simeq \Fun_{\C}(\M, \M') \]
which maps $\unit_{\C} \boxtimes_{\C^{\rev} \boe \C} f \mapsto I_{\C} \odot f$.
\end{lem}
\begin{proof}
The left $\C^{\rev} \boe \C$-action on $\Fun_{\E}(\M, \M')$ is given by
$(c_1 \boe c_2) \odot g = c_2 \odot g \odot c_1 = c_2 \odot g(c_1 \odot -)$
for $c_1 \boe c_2 \in \C^{\rev} \boe \C$, $g \in \Fun_{\E}(\M, \M')$.
The composed equivalences 
\[ \C^{\op|L} \boxtimes_{\C^{\rev} \boe \C} \Fun_{\E}(\M, \M') \simeq \Fun_{\C^{\rev} \boe \C}(\C, \Fun_{\E}(\M, \M')) \simeq \Fun_{\C}(\M, \M') \]
carries $\unit_{\C} \boxtimes_{\C^{\rev} \boe \C} f \mapsto [-, \unit_{\C}]^R_{\C^{\rev} \boe \C} \odot f \mapsto I_{\C} \odot f$.
Consider the composed equivalence 
\[ \C \boxtimes_{\C^{\rev} \boe \C} \Fun_{\E}({}^{LL}_{\C}\!\M, \M') \simeq \Fun_{\C^{\rev} \boe \C}(\C^{R|\op}, \Fun_{\E}({}^{LL}_{\C}\!\M, \M')) \simeq \Fun_{\C}(\M, \M') \]
The first equivalence holds since $(\C^{R|\op})^{\op|L} = \C$.
We want to check the second equivalence holds.  The left $\C^{\rev} \boe \C$-module structure on $\C^{R|\op}$ is defined as
\[ (b \boe a) \otimes^R x = x \otimes (b \boe a)^R = x \otimes (b^L \boe a^R) = b^L \otimes x \otimes a^R \]
for $b \boe a \in \C^{\rev} \boe \C$ and $x \in \C^{R|\op}$.
The left $\C^{\rev} \boe \C$-module structure on $\Fun_{\E}({}^{LL}_{\C}\!\M, \M')$ is defined as $(b \boe a) \odot g \coloneqq a \odot g(b \odot^{LL} m) = a \odot g(b^{LL} \odot m)$, for $g \in \Fun_{\E}({}^{LL}_{\C}\!\M, \M')$.
For a left $\C^{\rev} \boe \C$-module functor $f: \C^{R|\op} \to \Fun_{\E}({}^{LL}_{\C}\!\M, \M')$, $\unit_{\C} \mapsto f^{\unit_{\C}}$, check that $f^{\unit_{\C}}$ belongs to $\Fun_{\C}(\M, \M')$.
The left $\C$-module structure on $f^{\unit_{\C}}$ is induced by
$f^{\unit_{\C}}(c \odot m) = (c^{RR} \boe \unit_{\C}) \odot f^{\unit_{\C}}(m) \simeq f^{(c^{RR} \boe \unit_{\C}) \otimes^R \unit_{\C}}(m) = f^{c^R}(m) = f^{(\unit_{\C} \boe c) \otimes^R \unit_{\C}}(m) \simeq (\unit_{\C} \boe c) \odot f^{\unit_{\C}}(m) = c \odot f^{\unit_{\C}}(m)$
for $c \in \C, m \in \M$.
\end{proof}

\begin{rem}
\label{rem-1.30}
If $\N, \N' \in \RMod_{\C}(\cat)$,  we have the equivalence
\[ \C^{\rev} \boxtimes_{\C \boe \C^{\rev}} \Fun_{\E}({}^{LL}_{\C^{\rev}}\!\N, \N') \simeq \Fun_{\C^{\rev}}(\N, \N') \]
which maps $\unit_{\C^{\rev}} \boxtimes_{\C \boe \C^{\rev}} g \mapsto I_{\C^{\rev}} \odot g$, or equivalently,
\[ \Fun_{\E}(\N^{RR}_{\C}, \N') \boxtimes_{\C^{\rev} \boe \C} \C \simeq \Fun_{\C^{\rev}}(\N, \N') \]
which maps $g \boxtimes_{\C^{\rev} \boe \C} \unit_{\C} \mapsto I_{\C^{\rev}} \odot g$.
\end{rem}

By \cite[Lem.\,3.1.3]{Liang} and \cite[Prop.\,4.28]{Wei}, we have the following lemma.
\begin{lem}
\label{Lem-NDP}
Let $\C, \D$ be multifusion categories over $\E$, $\M \in \LMod_{\C}(\cat)$, $\N \in \BMod_{\C|\D}(\cat)$ and $\CP \in \LMod_{\D}(\cat)$. There is an equivalence
\[ \Fun_{\C}(\M, \N) \boxtimes_{\D} \CP \simeq \Fun_{\C}(\M, \N \boxtimes_{\D} \CP), \qquad f \boxtimes_{\D} y \mapsto f(-) \boxtimes_{\D} y. \]
\end{lem}

\begin{cor}
If $\D$ is replaced by $\D^{\rev}$, then there is an equivalence
\[ \CP \boxtimes_{\D} \Fun_{\C}(\M, \N) \simeq \Fun_{\C}(\M, \CP \boxtimes_{\D} \N) \]
\end{cor}
\begin{proof}
The left $\D$-action on $\Fun_{\C}(\M, \N)$ is induced by the right $\D^{\rev}$-action on $\N$.
Then we have
$\Fun_{\C}(\M, \N) \boxtimes_{\D^{\rev}} \CP = \CP \boxtimes_{\D} \Fun_{\C}(\M, \N)$.
The left $\C$-action on $\CP \boxtimes_{\D} \N$ is induced by the left $\C$-action on $\N$.
Then we have
$\Fun_{\C}(\M, \N \boxtimes_{\D^{\rev}} \CP) \simeq \M^{\op|L} \boxtimes_{\C} (\N \boxtimes_{\D^{\rev}} \CP) \simeq \M^{\op|L} \boxtimes_{\C} (\CP \boxtimes_{\D} \N)$.
\end{proof}

\begin{prop}
Let $\A, \B, \C, \D$ be multifusion categories over $\E$ and $\M \in \BMod_{\A|\C}(\cat)$, $\M' \in \BMod_{\A|\D}(\cat)$, $\N \in \BMod_{\C|\B}(\cat)$ and $\N' \in \BMod_{\D|\B}(\cat)$.
There is an equivalence
\begin{equation}
\label{eq-1}
\Fun_{\A}(\M, \M') \boxtimes_{\C^{\rev} \boe \D} \Fun_{\B^{\rev}}({}^{LL}_{\C}\!\N, \N') \simeq \Fun_{\A \boe \B^{\rev}}(\M \boxtimes_{\C} \N, \M' \boxtimes_{\D} \N')
\end{equation}
defined by $f \boxtimes_{\C^{\rev} \boe \D} g \mapsto I_{\C} \odot (f(-) \boxtimes_{\D} g(-))$. 
\end{prop}
\begin{proof}
Using the Lem.\,\ref{lem-1.30}, we reduce the equivalence (\ref{eq-1}) for special case $\A = \E = \B$.
By Lem.\,\ref{lem-1.30} and Lem.\,\ref{Lem-NDP}, the composed equivalence
\begin{align*}
& \Fun_{\E}(\M, \M') \boxtimes_{\C^{\rev} \boe \D} \Fun_{\E}({}^{LL}_{\C}\!\N, \N')\\
& \simeq \C \boxtimes_{\C \boe \C^{\rev}} (\Fun_{\E}(\M, \M') \boxtimes_{\D} \Fun_{\E}({}^{LL}_{\C}\!\N, \N') ) \\
& \simeq \C \boxtimes_{\C^{\rev} \boe \C} \Fun_{\E}({}^{LL}_{\C}\!\N, \Fun_{\E}(\M, \M') \boxtimes_{\D} \N') \\
& \simeq \Fun_{\C}(\N, \Fun_{\E}(\M, \M') \boxtimes_{\D} \N') \\
& \simeq \Fun_{\C}(\N, \Fun_{\E}(\M, \M' \boxtimes_{\D} \N')) \\
& \simeq \Fun_{\E}(\M \boxtimes_{\C} \N, \M' \boxtimes_{\D} \N')
\end{align*}
carries $f \boxtimes_{\C^{\rev} \boe \D} g \mapsto \unit_{\C} \boxtimes_{\C \boe \C^{\rev}} (f \boxtimes_{\D} g) \mapsto \unit_{\C} \boxtimes_{\C^{\rev} \boe \C} (f \boxtimes_{\D} g(-)) \mapsto I_{\C} \odot (f \boxtimes_{\D} g(-)) \mapsto I_{\C} \odot (f(-) \boxtimes_{\D} g(-)) \mapsto I_{\C} \odot (f(-) \boxtimes_{\D} g(-))$.
\end{proof}

\begin{rem}
The equivalence (\ref{eq-1}) induces an equivalence
\begin{equation}
\label{eq-2}
\Fun_{\A}(\M^{RR}_{\C}, \M') \boxtimes_{\C^{\rev} \boe \D} \Fun_{\B^{\rev}}(\N, \N') \simeq \Fun_{\A \boe \B^{\rev}}(\M \boxtimes_{\C} \N, \M' \boxtimes_{\D} \N')
\end{equation}
defined by $f \boxtimes_{\C^{\rev} \boe \D} g \mapsto I_{\C^{\rev}} \odot (f(-) \boxtimes_{\D} g(-))$.
\end{rem}

Theorem 3.1.7 in \cite{Liang} tells us the following fact. Let $\C, \D, \CP$ be multitensor categories and $\M, \M'$ be finite $\C$-$\D$ bimodules and $\N, \N'$ be finite $\D$-$\CP$ bimodules. Assume $\D$ is indecomposable. The assignment $f \boxtimes_{Z(\D)} g \mapsto f \boxtimes_{\D} g$ determines an equivalence of $Z(\C)$-$Z(\CP)$ bimodules
\begin{equation}
\label{eq-KZ}
 \Fun_{\C \boxtimes \D^{\rev}}(\M, \M') \boxtimes_{Z(\D)} \Fun_{\D \boxtimes \CP^{\rev}}(\N, \N') \simeq \Fun_{\C \boxtimes \CP^{\rev}}(\M \boxtimes_{\D} \N, \M' \boxtimes_{\D} \N') 
 \end{equation}
Moreover, when $\M = \M'$ and $\N = \N'$, (\ref{eq-KZ}) is an equivalence of monoidal $Z(\C)$-$Z(\CP)$ bimodules.
The statement and the proof idea of Thm.\,\ref{thm-Fun-braided-bim} comes from \cite[Thm.\,3.1.7]{Liang}.

\begin{thm}
\label{thm-Fun-braided-bim}
Let $\C, \D, \CP$ be fully faithful multifusion categories over $\E$.
Let $\M$, $\M' \in \BMod_{\C|\D}(\cat)$, $\N$, $\N' \in \BMod_{\D|\CP}(\cat)$. There is an equivalence of $Z(\C, \E)$-$Z(\CP, \E)$ bimodules
\begin{equation}
\label{eq-5555}
\Fun_{\C \boe \D^{\rev}}(\M, \M') \boxtimes_{Z(\D, \E)} \Fun_{\D \boe \CP^{\rev}}(\N, \N') \simeq \Fun_{\C \boe \CP^{\rev}}(\M \boxtimes_{\D} \N, \M' \boxtimes_{\D} \N')
\end{equation}
Moreover, when $\M = \M'$, $\N = \N'$, (\ref{eq-5555}) is an equivalence in $\BMod_{Z(\C, \E)|Z(\CP, \E)}(\Alg(\cat))$.
\end{thm}

\begin{proof}
We have the following composed equivalences:
\begin{align*}
& \Fun_{\C \boe \D^{\rev}}(\M, \M') \boxtimes_{Z(\D, \E)} \Fun_{\D \boe \CP^{\rev}}(\N, \N')  \\
& \simeq \Fun_{\C}(\M^{RR}_{\D}, \M') \boxtimes_{\D^{\rev} \boe \D} \D \boxtimes_{Z(\D, \E)} \D \boxtimes_{\D^{\rev} \boe \D} \Fun_{\CP^{\rev}}({}^{LL}_{\D}\!\N, \N') \\
& \simeq \Fun_{\C}(\M^{RR}_{\D}, \M') \boxtimes_{\D^{\rev} \boe \D} \Fun_{\E}(\D, \D) \boxtimes_{\D^{\rev} \boe \D} \Fun_{\CP^{\rev}}({}^{LL}_{\D}\!\N, \N') \\
& \simeq \Fun_{\C}(\M \boxtimes_{\D} \D, \M' \boxtimes_{\D} \D) \boxtimes_{\D^{\rev} \boe \D} \Fun_{\CP^{\rev}}({}^{LL}_{\D}\!\N, \N')  \\
& \simeq \Fun_{\C \boe \CP^{\rev}}(\M \boxtimes_{\D} \N, \M' \boxtimes_{\D} \N')
\end{align*}
The first step holds by Lem.\,\ref{lem-1.30} and Rem.\,\ref{rem-1.30}.
The second step holds by the equivalence $\C \boxtimes_{Z(\C, \E)} \C \simeq \Fun_{\E}(\C, \C)$ in \cite[Lem.\,5.23]{Wei}.
The third step and the last step hold by the equivalences (\ref{eq-2}) and (\ref{eq-1}).
This composed equivalence maps $f \boxtimes_{Z(\D, \E)} g$ to $f \boxtimes_{\D} g$. 
When $\M = \M'$ and $\N = \N'$, the formula $f \boxtimes_{Z(\D, \E)} g \mapsto f \boxtimes_{\D} g$ defines a monoidal equivalence. Check the diagram (\ref{braided-functor-E}) commutes.
\end{proof}

\begin{cor}
\label{cor-cylinder-c}
Taking $\C = \E, \M = \M' = \D$, $\CP = \E, \N = \N'$, we have the equivalence
\[ \D \boxtimes_{Z(\D, \E)} \Fun_{\D}(\N, \N) \simeq \Fun_{\E}(\N, \N) \]
\end{cor}

\begin{cor}
Taking $\D = \E$, we have $Z(\E, \E) = \E$ and the equivalence
\begin{equation}
\label{eq-m-pro}
 \Fun_{\C}(\M, \M') \boe \Fun_{\CP^{\rev}}(\N, \N') \simeq \Fun_{\C \boe \CP^{\rev}}(\M \boe \N, \M' \boe \N') 
 \end{equation}
\end{cor}
\begin{cor}
\label{cor-Fun-braided-bim}
Let $\C_1, \C_2, \D_1, \D_2$ be fully faithful multifusion categories over $\E$.
Let $\M, \M' \in \LMod_{\C_1 \boe \D_1^{\rev}}(\cat)$ and $\N, \N' \in \LMod_{\C_2 \boe \D_2^{\rev}}(\cat)$.
Taking $\C = \C_1 \boe \D_1^{\rev}$ and $\CP^{\rev} = \C_2 \boe \D_2^{\rev}$ in (\ref{eq-m-pro}),  we have the equivalence
\[ \Fun_{\C_1 \boe \D_1^{\rev}}(\M, \M') \boe \Fun_{\C_2 \boe \D_2^{\rev}}(\N, \N') \simeq \Fun_{\C_1 \boe \C_2 \boe \D_1^{\rev} \boe \D_2^{\rev}}(\M \boe \N, \M' \boe \N')  \]
Taking $\C_1 = \D_1 = \M = \M' = \C$ and $\C_2 = \D_2 = \N = \N' = \D$ in the above equivalence, we have the braided monoidal equivalence:
\[ Z(\C, \E) \boe Z(\D, \E) \simeq Z(\C \boe \D, \E) \]
by the equivalence $\Fun_{\C \boe \C^{\rev}}(\C, \C) \simeq Z(\C, \E)$ (see \cite[Thm.\,4.18]{Wei}).
\end{cor}

We introduce two symmetric monoidal categories $\MFus$ and $\BFus$.
\begin{itemize}
\item The category $\MFus$ consists of fully faithful multifusion categories over $\E$, and the equivalence classes of  bimodules in $\cat$ (see Def.\,\ref{defn-CD-bimodule}). 
The category $\MFus$ with the tensor product $\boe$ and the tensor unit $\E$ is a symmetric monoidal category.
\item The category $\BFus$ consists of braided fusion categories over $\E$, and the equivalence classes of bimodules in $\Alg(\cat)$ (see Def.\,\ref{defn-cd-braided-bim}). 
The category $\BFus$ with the tensor product $\boe$ and the unit $\E$ is a symmetric monoidal category. 
\end{itemize}

\begin{thm}
\label{thm-func-E}
The assignment 
\[ \C \mapsto Z(\C, \E), \qquad \M \mapsto  \Fun_{\C \boe \D^{\rev}}(\M, \M) \]
defines a symmetric monoidal functor $\mathfrak{Z}_{/\E}: \MFus \to \BFus$.
\end{thm}

\begin{proof}
By Expl.\,\ref{expl-braided-bimodule}, $\FZ_{/\E}$ is well defined on morphism. 
By Thm.\,\ref{thm-Fun-braided-bim}, $\FZ_{/\E}$ is a well-defined functor.
By Cor.\,\ref{cor-Fun-braided-bim}, $\FZ_{/\E}$ is a symmetric monoidal functor.
\end{proof}

\subsection{Exercises of Monoidal functors over $\E$}

\begin{expl}
Let $\C, \D$ be multifusion categories over $\E$ and $f: \C \to \D$ a monoidal functor over $\E$. The $\C$-$\C$ bimodule structure on $\D$ is denoted by ${}_f\D_f$.
\begin{itemize}
\item $\C$ belongs to $\BMod_{\C|\C}(\cat)$. The left $\E$-module on $\C$ is defined as $e \odot x \coloneqq T_{\C}(e) \otimes x$ for $e \in \E$, $x \in \C$.
The monoidal natural isomorphism is defined as $v^{\C}_e: T_{\C}(e) \otimes - \xrightarrow{z_{e,-}} - \otimes T_{\C}(e)$ for all $e \in \E$, $(T_{\C}(e), z_{e,-}) \in Z(\C)$.
 \item ${}_f\D_f$ belongs to $\BMod_{\C|\C}(\cat)$. The left $\E$-module structure on $\D$ is defined as $e \odot d \coloneqq T_{\D}(e) \otimes d$ for $e \in \E$, $d \in \D$. The $\C$-$\C$ bimodule structure on $\D$ is defined as $c \odot d \odot \tilde{c} \coloneqq f(c) \otimes d \otimes f(\tilde{c})$ for $c, \tilde{c} \in \C$, $d \in \D$. 
The monoidal natural isomorphism is defined as 
\[ v^{\D}_e: T_{\C}(e) \odot d = f(T_{\C}(e)) \otimes d \simeq T_{\D}(e) \otimes d \xrightarrow{\hat{z}_{e,d}} d \otimes T_{\D}(e) \simeq d \otimes f(T_{\C}(e)) = d \odot T_{\C}(e) \]
for all $e \in \E$, $(T_{\D}(e), \hat{z}_{e,-}) \in Z(\D)$.
\item ${}_f\D$ belongs to $\BMod_{\C|\D}(\cat)$. The left $\E$-module structure on $\D$ is defined as $e \odot d \coloneqq T_{\D}(e) \otimes d$ for $e \in \E$, $d \in \D$. The $\C$-$\D$ bimodule structure on $\D$ is defined as $c \odot d \odot d' \coloneqq f(c) \otimes d \otimes d'$ for $c \in \C, d, d' \in \D$. The monoidal natural isomorphism is defined as 
\[ \bar{v}^{\D}_e: T_{\C}(e) \odot d = f(T_{\C}(e)) \otimes d \simeq T_{\D}(e) \otimes d \xrightarrow{\hat{z}_{e,d}} d \otimes T_{\D}(e) \] 
for all $e \in \E$, $(T_{\D}(e), \hat{z}_{e,-}) \in Z(\D)$.
\end{itemize}
\end{expl}

\begin{lem}
\label{lem-CD-DD}
Let $\C, \D$ be multifusion categories over $\E$ and $f: \C \to \D$ a monoidal functor over $\E$. There is a monoidal equivalence
\begin{equation}
\label{eq-moe}
 \Fun_{\C \boe \C^{\rev}}(\C, \D) \simeq \Fun_{\C \boe \D^{\rev}}(\D, \D) 
 \end{equation}
\end{lem}
\begin{proof}
We have the composed equivalences
$\Fun_{\C \boe \D^{\rev}}(\D, \D) \simeq \Fun_{\C \boe \C^{\rev}}(\C, \Fun_{\D^{\rev}}(\D, \D)) \simeq \Fun_{\C \boe \C^{\rev}}(\C, \D)$ of categories by Prop.\,\ref{prop-3.18} and Prop.\,\ref{prop-C-3.16}. 
Since the equivalence $\Fun_{\C \boxtimes \C^{\rev}}(\C, \D) = \Fun_{\C \boxtimes \D^{\rev}}(\D, \D)$ holds as monoidal categories by \cite[Sec.\,3.2]{Liang}, (\ref{eq-moe}) is a monoidal equivalence.
\end{proof}

\begin{lem}
Let $\C, \D$ be fully faithful multifusion categories over $\E$ and $f: \C \to \D$ be a monoidal functor over $\E$.
There is an equivalence of right $Z(\D, \E)$-modules
\begin{equation}
\label{eq-}
\C \boxtimes_{Z(\C, \E)} \Fun_{\C \boe \C^{\rev}}(\C, {}_f\D_f) \simeq \D
\end{equation}
\end{lem}

\begin{proof}
We have the composed equivalence
$\Fun_{\C^{\rev}}(\C, \C) \boxtimes_{Z(\C, \E)} \Fun_{\C \boe \C^{\rev}}(\C, {}_f\D_f) \simeq \Fun_{\E \boe \C^{\rev}}(\C \boxtimes_{\C} \C, \C \boxtimes_{\C} {}_f\D_f)  \simeq \Fun_{\C^{\rev}}(\C, {}_f\D_f) \simeq \D$ by 
Prop.\,\ref{prop-C-3.16} and Thm.\,\ref{thm-Fun-braided-bim}.
\end{proof}


\begin{prop}
\label{prop-closed-bimodule}
Let $\C, \D, \B$ be $\NBFCe$'s.
If $\M \in \BMod_{\C|\D}(\Alg(\cat))$ and $\N \in \BMod_{\D|\B}(\Alg(\cat))$ are closed, then $\M \boxtimes_{\D} \N \in \BMod_{\C|\B}(\Alg(\cat))$ is closed.
\end{prop}

\begin{proof}
$\C^{\rev} \boe \N$ is a closed in $\BMod_{\overline{\C} \boe \D | \overline{\C} \boe \B}(\Alg(\cat))$ by the composed equivalences
\[ Z(\C^{\rev} \boe \N, \E) \cong Z(\C^{\rev}, \E) \boe Z(\N, \E) \simeq \overline{\C^{\rev}} \boe \C^{\rev} \boe \overline{\D} \boe \B \simeq  (\C \boe \overline{\D}) \boe (\overline{\C} \boe \B)  \]
The first and second equivalences hold by \cite[Prop.\,4.7, Cor.\,4.4]{DNO}.
And we have $\M \boxtimes_{\D} \N \simeq \M \boxtimes_{\overline{\C} \boe \D} (\C^{\rev} \boe \N)$.
It is enough to prove the special case when $\C = \E$ and $\D = Z(\M, \E)$.

Let $\CP = \M \boxtimes_{Z(\M, \E)} \N$. 
Since $\D = Z(\M, \E)$ is a braided fusion category over $\E$, the multifusion category $\M$ is indecomposable.
By  Lem.\,\ref{lem-CD-DD}, Lem.\,\ref{Lem-NDP} and \cite[Thm.\,4.18]{Wei}, 
the equivalences 
$\Fun_{\M \boe \CP^{\rev}}(\CP, \CP) \simeq \Fun_{\M \boe \M^{\rev}}(\M, \M \boxtimes_{Z(\M, \E)} \N) \simeq \Fun_{\M \boe \M^{\rev}}(\M, \M) \boxtimes_{Z(\M, \E)} \N \simeq Z(\M, \E) \boxtimes_{Z(\M, \E)} \N \simeq \N$
hold in $\BMod_{Z(\M, \E)|Z(\CP, \E)}(\Alg(\cat))$.
Since the module category $\CP$ over $\M \boe \CP^{\rev}$ is faithful, the following equivalence holds
\[ \Fun_{\Fun_{\M \boe \CP^{\rev}}(\CP, \CP)}(\CP, \CP) \simeq \M \boe \CP^{\rev} \]
Then $\M^{\rev} \boe \CP$ is Morita equivalent to $\Fun_{\M \boe \CP^{\rev}}(\CP, \CP)$ in $\cat$.
Then we have $Z(\M^{\rev} \boe \CP, \E) \simeq Z(\Fun_{\M \boe \CP^{\rev}}(\CP, \CP), \E) \simeq Z(\N, \E)$.
By the equivalence
$Z(\M^{\rev}, \E) \boe Z(\CP, \E) \simeq Z(\M^{\rev} \boe \CP, \E) \simeq Z(\N, \E) \simeq \overline{Z(\M, \E)} \boe \B$, we conclude the canonical functor $\B \to Z(\CP, \E)$ is an equivalence.
\end{proof}

\section{Integrate on a closed stratified surface}
Recall Section 5.2 in \cite{Wei}.

\begin{figure}[t]
\centering
\begin{subfigure}{0.49\linewidth}
\centering
\begin{tikzpicture}
\fill[green!10!white] (-1,-1) -- (1,-1) -- (1,1) -- (-1,1) -- cycle;
\fill[gray!20!white] (-3,-1) -- (-1,-1) -- (-1,1) -- (-3,1) -- cycle;
\fill[blue!20!white] (1,-1) -- (3,-1) -- (3,1) -- (1,1) -- cycle;
\draw [->, line width=1pt] (-1,1) -- (-1,-0.1); \draw[-, line width=1pt](-1, -0.1) -- (-1, -1);
\draw [->,line width=1pt] (1,-1)--(1, 0.1); \draw[-, line width=1pt](1, 0.1)--(1,1);
\draw (-1.25,1.6) node[anchor=north west] {$\M$};
\draw (0.77,1.6) node[anchor=north west] {$\N$};
\draw [color=blue](-2.1052648619193723,0.3087148761534759) node[anchor=north west] {$\C$};
\draw [color=blue](-0.12038188820630064,0.3087148761534759) node[anchor=north west] {$\D$};
\draw [color=blue](1.9380152697183661,0.3087148761534759) node[anchor=north west] {$\B$};
\end{tikzpicture}
\caption{}
\label{fig:sub1}
\end{subfigure}
\begin{subfigure}{0.49\linewidth}
\centering
\begin{tikzpicture}
\fill[blue!20!white] (-1,-1) -- (1,-1) -- (1,1) -- (-1,1) -- cycle;
\fill[gray!20!white] (-3,-1) -- (-1,-1) -- (-1,1) -- (-3,1) -- cycle;
\draw [->, line width=1pt] (-1,1) -- (-1,-0.1); \draw[-, line width=1pt](-1, -0.1) -- (-1, -1);
\draw (-1.7,1.6) node[anchor=north west] {$\M\otimes_{\D} \N^{\rev}$};
\draw [color=blue](-2.1052648619193723,0.3087148761534759) node[anchor=north west] {$\C$};
\draw [color=blue](-0.12038188820630064,0.3087148761534759) node[anchor=north west] {$\B$};
\end{tikzpicture}
\caption{}
\label{fig:sub2}
\end{subfigure}
\caption{Fig.\,\ref{fig:sub1} depicts a stratified 2-disk $K=(\Rb^2; \Rb \cup \Rb)$ with a coefficient system $A_K$ determined by 2-disk algebras $\C, \D, \B$ for 2-cells and 1-disk algebras $\M, \N$ for 1-cells. Fig.\,\ref{fig:sub2} depicts a stratified 2-disk $K'=(\Rb^2; \Rb)$ with a coefficient system $A_{K'}$ determined by 2-disk algebras $\C, \B$ for 2-cells and a 1-disk algebra $\M \otimes_{\D} \N^{\rev}$ for the 1-cell. 
$K'$ is obtained by fusing two 1-cells $\M, \N$ and a 2-cell $\D$ to a 1-cell $\M \otimes_{\D} \N^{\rev}$.}
\label{11}
\end{figure}

\begin{expl}[\cite{LiangFH}\,Expl.\,3.22]
\label{expl3.22}
 There are a few examples using the $\otimes$-excision property.
\begin{itemize}
\item[(1)] A stratified 2-disk $K = (\Rb^2; \Rb \cup \Rb)$ with a coefficient system $A_K$ is shown in Fig.\,\ref{fig:sub1}. The 1-disk algebra $\M$ is a $\C$-$\D$ bimodule, and $\N$ is a $\B$-$\D$ bimodule. Note that our convention of left and right is that if one stands on the 1-cell seeing the arrow pointing towards you, then the left hand side of you is treated as the left. By the $\otimes$-excision property, we have
\[ \int_{K} A_K \simeq \M \otimes_{\D} \N^{\rev} \]
Consider a process of fusing these two 1-cells $\M$, $\N$ into a 1-cell $\M \otimes_{\D} \N^{\rev}$. This process produces a new stratified 2-disk $K'$ and a new coefficient system $A_{K'}$ shown in Fig.\,\ref{fig:sub2}.
We have $\int_{K} A_{K} \simeq \int_{K'} A_{K'}$.
\item[(2)] Let $M$ be a stratified 2-disk with a coefficient system $A_M$ as shown in Fig.\,\ref{fig:sub11}. 
In particular, $\CP, \CQ$ are 0-disk algebras in $\V$; $\M_1, \dots, \M_m, \N_1, \dots, \N_n, \CL$ are 1-disk algebras; and $\C_0, \dots, \C_m = \D_0, \dots, \D_n = \C_0$ are 2-disk algebras. By the $\otimes$-excision property, we have
\[ \int_M A_M \simeq \CP \otimes_{\CL} \CQ \]

Consider a process of contracting the internal edge $\CL$ to a point. It produces a new stratfied 2-disk $M'$ with a new coefficient system $A_{M'}$ depicted in Fig.\,\ref{fig:sub22}. We have
\[ \int_{M} A_M \simeq \CP \otimes_{\CL} \CQ \simeq \int_{M'} A_{M'} \]
\item[(3)] Let $N$ be a stratified 2-disk with a coefficient system $A_N$ as depicted in Fig.\,\ref{fig-3-1}. By the same argument as the previous case, we have
\[ \int_N A_N \simeq \CP \otimes_{\CK \otimes_{\A} \CL} \CQ \]
Therefore, $\int_N A_N \simeq \int_{N'} A_{N'}$, where $N'$ is a stratified 2-disk with a cofficient system $A_{N'}$ as depicted in Fig.\,\ref{fig:3-2}.
\item[(4)] The $\otimes$-excision property also allows us to add 0-cells and 1-cells with proper target labels without changing the value of factorization homology.
\begin{itemize}
\item[(a)] On any 1-cell $e$ with the target label $\M$, we can add a new 0-cell labelled by the 0-disk algebra $\M$, which is obtained by forgetting its 1-disk algebra structure. Now $e$ breaks into two 1-cells both labelled by $\M$ and connected by a 0-cell labelled by the 0-disk algebra $\M$. 
\item[(b)] Between any two 0-cells $p$ and $q$ (not necessarily distinct) on the boundary of a given 2-cell labelled by $\C$, we can add an oriented 1-cell within this 2-cell from $p$ to $q$ labelled by the 1-disk algebra $\C$ obtained by forgetting its 2-disk algebra structure.
\end{itemize}
\end{itemize}
\end{expl}

Recall Section 5.1 and 5.4 in \cite{Wei}.

\begin{figure}[t]
\centering
\begin{subfigure}{0.49\linewidth}
\centering
\begin{tikzpicture}
\fill[blue!20!white] (-1.815083985599556,-0.9847888080177889) -- (2.5,-1) -- (2.5,1) -- (-1.823117296027546,1.01550548855167) -- cycle;
\draw [color=blue](-0.7381859530143546,0.28200598688042167);
\node [color=blue] at (-1.1, 0) {\small{$\CP$}};
\draw [color=blue](1.1742735973498657,0.2912897711054906);
\node [color=blue] at (1.7, 0) {\small{$\CQ$}};
\draw [->,line width=1pt] (-1.4113432011203064,-0.9995834173491595) -- (-1.1151377593651932,-0.5058988487332496);
\draw [->,line width=1pt] (-1.3912139246213802,0.9932149560445124) -- (-1.1006106189296851,0.4886650845489827);
\node [color=blue] at (-0.9,0.8){\small{$\M_m$}};
\node [color=blue] at (-0.8,-0.73) {\small{$\M_1$}};
\node [color=blue] at (1.4,0.8){\small{$\N_1$}};
\node [color=blue] at (1.38,-0.72){\small{$\N_n$}};
\node [color=blue] at (0.5,0.2){\small{$\CL$}};
\node [color=blue] at (0.2,0.65){\small{$\C_m=\D_0$}};
\node [color=blue] at (0.3,-0.54) {\small{$\C_0=\D_n$}};
\draw [line width=1pt] (-1.3912139246213802,0.9932149560445124)-- (-0.8076319431304017,-0.003539090793807049);
\draw [line width=1pt] (-1.4113432011203064,-0.9995834173491595)-- (-0.8076319431304017,-0.003539090793807049);
\draw [->,line width=1pt] (-0.8076319431304017,-0.003539090793807049) -- (0.3877210541306687,-0.003539090793807715);
\draw [line width=1pt] (-0.8076319431304017,-0.003539090793807049)-- (1.3838485518482282,0);
\draw [line width=1pt] (1.3838485518482282,0)-- (2,1);
\draw [line width=1pt] (1.3838485518482282,0)-- (2,-1);
\draw [->,line width=1pt] (2,1) -- (1.6956069998154117,0.5059769783911805);
\draw [->,line width=1pt] (2,-1) -- (1.698768131410779,-0.5111074241685122);
\begin{scriptsize}
\draw [fill=black] (-1.5321188601138636,0.45978912882297396) circle (1.5pt);
\draw [fill=black] (-1.6528945191074207,0.027009684096065443) circle (1.5pt);
\draw [fill=black] (-1.5220542218644004,-0.3957051223813801) circle (1.5pt);
\draw [fill=black] (2.098716520798477,0.4769459375170115) circle (1.5pt);
\draw [fill=black] (2.3565377555018454,0.031618350302105885) circle (1.5pt);
\draw [fill=black] (2.1221548148624194,-0.4137092369127997) circle (1.5pt);
\draw [fill=black] (-0.8076319431304017,-0.003539090793807049) circle (2pt);
\draw [fill=black] (1.4072868459121708,-0.003539090793807715) circle (2pt);
\end{scriptsize}
\end{tikzpicture}
\caption{}
\label{fig:sub11}
\end{subfigure}
\begin{subfigure}{0.49\linewidth}
\centering
\begin{tikzpicture}
\fill[blue!20!white] (-2.3,1) -- (-2.3,-1) -- (0.8,-1) -- (0.8,1) -- cycle;
\node [color=blue] at (-0.2,0){\small{$\CP \otimes_{\CL} \CQ$}};
\node [color=blue] at (-1.68,0.8){\small{$\M_m$}};
\node [color=blue] at (-1.65,-0.78) {\small{$\M_1$}};
\node [color=blue] at (0.2,0.8){\small{$\N_1$}};
\node [color=blue] at (0.2,-0.78){\small{$\N_n$}};
\node [color=blue] at (-0.8,0.8){\small{$\D_0$}};
\node [color=blue] at (-0.8,-0.8){\small{$\C_0$}};
\draw [line width=1pt] (-1.5,1)-- (-0.8088793981429308,0);
\draw [line width=1pt] (-0.8088793981429308,0)-- (-1.5,-1);
\draw [line width=1pt] (-0.8088793981429308,0)-- (-0.10423593792685772,1.00598350092124);
\draw [line width=1pt] (-0.8088793981429308,0)-- (-0.09697157235761986,-0.9917170306191709);
\draw [->,line width=1pt] (-1.5,1) -- (-1.2159280990646029,0.5889691319111537);
\draw [->,line width=1pt] (-1.5,-1) -- (-1.2220458840232147,-0.5978211107729805);
\draw [->,line width=1pt] (-0.09697157235761986,-0.9917170306191709) -- (-0.38109289177148986,-0.5959242874871457);
\draw [->,line width=1pt] (-0.1042359379268577,1.00598350092124) -- (-0.3493808336723666,0.6560026462356404);
\begin{scriptsize}
\draw [fill=black] (-1.5321188601138636,0.45978912882297396) circle (1.5pt);
\draw [fill=black] (-1.6528945191074207,0.027009684096065443) circle (1.5pt);
\draw [fill=black] (-1.5220542218644004,-0.3957051223813801) circle (1.5pt);
\draw [fill=black] (0.1,0.5) circle (1.5pt);
\draw [fill=black] (0.45,0) circle (1.5pt);
\draw [fill=black] (0.1,-0.47) circle (1.5pt);
\draw [fill=black] (-0.8088793981429308,0) circle (2pt);
\end{scriptsize}
\end{tikzpicture}
\caption{}
\label{fig:sub22}
\end{subfigure}
\caption{These two figures depict a process of contracting a 1-cell to a 0-cell with proper new target labels such that the value of factorization homology is not changed. Fig.\,\ref{fig:sub11} depicts a stratified 2-disk $M$ with a coefficient system $A_M$ determined by its target labels. Fig.\,\ref{fig:sub22} depicts another stratified 2-disk $M'$ with a coefficient system $A_{M'}$ determined by its target labels. $M'$ is obtained by contracting the internal edge $\CL$ in Fig.\,\ref{fig:sub11} to a point.}
\label{22}
\end{figure}

\begin{figure}[t]
\centering
\begin{subfigure}{0.49\linewidth}
\centering
\begin{tikzpicture}
\fill[fill=blue!20!white] (-1.7944766176864437,-1.0067998238684739) -- (2.5,-1) -- (2.5,1) -- (-1.8011833499349705,0.9985131184410266) -- cycle;
\node [color=blue] at (-1.1,0) {\small{P}};
\node [color=blue] at (1.8, 0) {\small{Q}};
\draw [->,line width=1pt] (-1.4113432011203064,-0.9995834173491595) -- (-1.1151377593651932,-0.5058988487332496);
\draw [->,line width=1pt] (-1.3912139246213802,0.9932149560445124) -- (-1.1006106189296851,0.4886650845489827);
\node [color=blue] at (-0.8,0.7){\small{$\M_m$}};
\node [color=blue] at (-0.8,-0.7){\small{$\M_1$}};
\node [color=blue] at (1.5,0.7) {\small{$\N_1$}};
\node [color=blue] at (1.46,-0.7) {\small{$\N_n$}};
\node [color=blue] at (0.26,0.54) {\small{$\CL$}};
\draw [line width=1pt] (-1.3912139246213802,0.9932149560445124)-- (-0.8076319431304017,-0.003539090793807049);
\draw [line width=1pt] (-1.4113432011203064,-0.9995834173491595)-- (-0.8076319431304017,-0.003539090793807049);
\draw [line width=1pt] (1.4487398362157924,0)-- (2,1);
\draw [line width=1pt] (1.4487398362157924,0)-- (2,-1);
\draw [->,line width=1pt] (2,1) -- (1.7276647881947529,0.5059769783911805);
\draw [->,line width=1pt] (2,-1) -- (1.7304929985742508,-0.5111074241685122);
\draw [shift={(0.3241158792485071,-2.272704343505922)},line width=1pt]  plot[domain=1.1112856068093877:2.0334440190860974,variable=\t]({1*2.535737383400623*cos(\t r)+0*2.535737383400623*sin(\t r)},{0*2.535737383400623*cos(\t r)+1*2.535737383400623*sin(\t r)});
\draw [shift={(0.31621103586228405,2.7670831544684664)},line width=1pt]  plot[domain=4.32703999864502:5.100874934430051,variable=\t]({1*2.98987803553478*cos(\t r)+0*2.98987803553478*sin(\t r)},{0*2.98987803553478*cos(\t r)+1*2.98987803553478*sin(\t r)});
\draw [->,line width=1pt] (0.23728661977429052,0.26154599086805175) -- (0.41874331882712473,0.2612667939094657);
\draw [->,line width=1pt] (0.24297419356040087,-0.22189778095132393) -- (0.40749884711286244,-0.22140094334607507);
\node [color=blue] at (0,0) {\small{$\A$}};
\node [color=blue] at (0.3,-0.5) {\small{$\CK$}};
\begin{scriptsize}
\draw [fill=black] (-1.5321188601138636,0.45978912882297396) circle (1.5pt);
\draw [fill=black] (-1.6528945191074207,0.027009684096065443) circle (1.5pt);
\draw [fill=black] (-1.5220542218644004,-0.3957051223813801) circle (1.5pt);
\draw [fill=black] (2.098716520798477,0.4769459375170115) circle (1.5pt);
\draw [fill=black] (2.3565377555018454,0.031618350302105885) circle (1.5pt);
\draw [fill=black] (2.1221548148624194,-0.4137092369127997) circle (1.5pt);
\draw [fill=black] (-0.8076319431304017,-0.003539090793807049) circle (2pt);
\draw [fill=black] (1.4487398362157924,0) circle (2pt);
\end{scriptsize}
\end{tikzpicture}
\caption{}
\label{fig-3-1}
\end{subfigure}
\begin{subfigure}{0.49\linewidth}
\centering
\begin{tikzpicture}
\fill[blue!20!white] (-1.815083985599556,-0.9847888080177889) -- (2.5,-1) -- (2.5,1) -- (-1.823117296027546,1.01550548855167) -- cycle;
\draw [color=blue](-0.7381859530143546,0.28200598688042167);
\node [color=blue] at (-1.1, 0) {\small{P}};
\draw [color=blue](1.1742735973498657,0.2912897711054906);
\node [color=blue] at (1.7, 0) {\small{Q}};
\draw [->,line width=1pt] (-1.4113432011203064,-0.9995834173491595) -- (-1.1151377593651932,-0.5058988487332496);
\draw [->,line width=1pt] (-1.3912139246213802,0.9932149560445124) -- (-1.1006106189296851,0.4886650845489827);
\node [color=blue] at (-0.9,0.8){\small{$\M_m$}};
\node [color=blue] at (-0.8,-0.73) {\small{$\M_1$}};
\node [color=blue] at (1.4,0.8){\small{$\N_1$}};
\node [color=blue] at (1.38,-0.72){\small{$\N_n$}};
\node [color=blue] at (0.35,0.25){\small{$\CK \otimes_{\A} \CL$}};
\draw [line width=1pt] (-1.3912139246213802,0.9932149560445124)-- (-0.8076319431304017,-0.003539090793807049);
\draw [line width=1pt] (-1.4113432011203064,-0.9995834173491595)-- (-0.8076319431304017,-0.003539090793807049);
\draw [->,line width=1pt] (-0.8076319431304017,-0.003539090793807049) -- (0.3877210541306687,-0.003539090793807715);
\draw [line width=1pt] (-0.8076319431304017,-0.003539090793807049)-- (1.3838485518482282,0);
\draw [line width=1pt] (1.3838485518482282,0)-- (2,1);
\draw [line width=1pt] (1.3838485518482282,0)-- (2,-1);
\draw [->,line width=1pt] (2,1) -- (1.6956069998154117,0.5059769783911805);
\draw [->,line width=1pt] (2,-1) -- (1.698768131410779,-0.5111074241685122);
\begin{scriptsize}
\draw [fill=black] (-1.5321188601138636,0.45978912882297396) circle (1.5pt);
\draw [fill=black] (-1.6528945191074207,0.027009684096065443) circle (1.5pt);
\draw [fill=black] (-1.5220542218644004,-0.3957051223813801) circle (1.5pt);
\draw [fill=black] (2.098716520798477,0.4769459375170115) circle (1.5pt);
\draw [fill=black] (2.3565377555018454,0.031618350302105885) circle (1.5pt);
\draw [fill=black] (2.1221548148624194,-0.4137092369127997) circle (1.5pt);
\draw [fill=black] (-0.8076319431304017,-0.003539090793807049) circle (2pt);
\draw [fill=black] (1.4072868459121708,-0.003539090793807715) circle (2pt);
\end{scriptsize}
\end{tikzpicture}
\caption{}
\label{fig:3-2}
\end{subfigure}
\caption{These two figures depict a process of merging two 1-cells and one 2-cell to a single 1-cell with proper new target labels such that the value of factorization homology is not changed.}
\end{figure}

\begin{lem}
For any stratified surface $\Sigma$ with an anomaly-free coefficient system $A$ in $\cat$ (see Def.\,5.26), any one of the processes described in Expl.\,\ref{expl3.22}(1), (2), (3), (4) produces a new stratified surface $\Sigma'$ and a new anomaly-free coefficient system $A'$ in $\cat$. We have
\begin{equation}
\label{eq-4.8}
 \int_{\Sigma} A \simeq \int_{\Sigma'} A'. 
 \end{equation}
\end{lem}

\begin{proof}
Equivalence (\ref{eq-4.8}) holds by the $\otimes$-excision property of factorization homology.
$\M \boxtimes_{\D} \N^{\rev}$ in Fig.\,\ref{fig:sub2} is anomaly-free in $\cat$ by Prop.\,\ref{prop-closed-bimodule}.
Check that $\CP \boxtimes_{\CL} \CQ$ in Fig.\,\ref{fig:sub22} is anomaly-free in $\cat$.
Let $\M = \M_1 \boxtimes_{\C_1} \cdots \boxtimes_{\C_{m-1}} \M_m$ and $\N = \N_1 \boxtimes_{\D_1} \cdots \boxtimes_{\D_{n-1}} \N_n$. We have $Z(\M, \E) \simeq Z(\CL, \E) \simeq Z(\N^{\rev}, \E) \simeq \overline{\C}_0 \boe \D_0$ by Prop.\,\ref{prop-closed-bimodule}.
Then $\M$ and $\CL$ are Morita equivalent in $\cat$ and $\N^{\rev}$ and $\CL$ are Morita equivalent in $\cat$.
That is, $\M \simeq \Fun_{\CL^{\rev}}(\CP, \CP)$ and $\N \simeq \Fun_{\CL}(\CQ, \CQ)$ by Def.\,\ref{defn-invertible}.
Then $\M \boxtimes_{\overline{\C}_0 \boe \D_0} \N \simeq \M \boxtimes_{Z(\CL, \E)} \N \simeq \Fun_{\CL^{\rev}}(\CP, \CP) \boxtimes_{Z(\CL, \E)} \Fun_{\CL}(\CQ, \CQ) \simeq \Fun_{\E}(\CP \boxtimes_{\CL} \CQ, \CP \boxtimes_{\CL} \CQ)$ by Thm.\,\ref{thm-Fun-braided-bim}.
It is routine to check that the new coefficients $A'$ in Expl.\,\ref{expl3.22} (3), (4) are anomaly-free in $\cat$.
\end{proof}

\begin{thm}
Given any stratified sphere $\Sigma = (S^2; \Gamma)$ and an anomaly-free coefficient system $A$ in $\cat$ on $\Sigma$, we have $\int_{\Sigma} A = (\E, u_{\Sigma})$, where $u_{\Sigma}$ is an object in $\E$.
\end{thm}
\begin{proof}
Applying the processes described in Expl.\,\ref{expl3.22} (2), (3), (4) repeatedly, we can reduced the graph $\Gamma$ to finitely many points on $S^2$ because all loops on $S^2$ are contractible. 
Then the result follows from \cite[Thm.\,5.29]{Wei}.
\end{proof}

\begin{thm}
\label{main-result}
Given any closed stratified surface $\Sigma$ and an anomaly-free coefficient system $A$ in $\cat$ on $\Sigma$, we have $\int_{\Sigma} A \simeq (\E, u_{\Sigma})$, where $u_{\Sigma}$ is an object in $\E$. 
\end{thm}
\begin{proof}
If $\Sigma$ is a stratified sphere, the result is true. We assume that the genus of $\Sigma$ is greater than zero. Apply the process described by Expl.\,\ref{expl3.22} (2), (3), (4) to the graph $\Gamma$ until no further reduction as possible.
Using \cite[Lem.\,5.32]{Wei} repeatedly, the genus is reduced to zero. 
\end{proof}



\begin{thebibliography}{KWak2}
\bibitem[AF1]{AF}
David Ayala, John Francis. Factorization homology of topological manifolds. Journal of Topology, 2015, 8(4):1045-1084.
\bibitem[AF2]{Francis}
David Ayala, John Francis. A factorization homology primer. In: Handbook of Homotopy theory. Chapman and Hall/CRC, 2020, 39-101.
\bibitem[AFR]{AFR}
David Ayala, John Francis, Nick Rozenblyum. Factorization homology i: Higher categories. Advances in Mathematics, 2018, 333:1024-1177. 
\bibitem[AFT1]{AFT1}
David Ayala, John Francis, Hiro Lee Tanaka. Local structures on stratified spaces. Advances in Mathematics, 2017, 307:903-1028.
\bibitem[AFT2]{AFT2}
David Ayala, John Francis, Hiro Lee Tanaka. Factorization homology of stratified spaces. Selecta Mathematica, 2016, 23(1):293-362.
\bibitem[AKZ]{LiangFH}
Yinghua Ai, Liang Kong, Hao Zheng. Topological orders and factorization homology. Advances in Theoretical and Mathematical Physics, 2017, 21(8):1854-1894.
\bibitem[BBJ1]{BBJ}
David Ben-Zvi, Adrien Brochier, David Jordan. Integrating Quantum groups over surfaces. Journal of Topology, 2018, 11(4):874-917.
\bibitem[BBJ2]{BBJ2}
David Ben-Zvi, Adrien Brochier, David Jordan. Quantum character varieties and braided module categories. Selecta Mathematica, 2018, 24(5):4711-4748. 
\bibitem[BD]{BD}
Alexander Beilinson, Vladimir Drinfeld. Chiral algebras. American Mathematical Society, Providence, R.I, 2004.
\bibitem[CG]{CG}
Kevin Costello, Owen Gwilliam. Factorization algebras in perturbative quantum field theory. Cambridge University Press, 2016.
\bibitem[DMNO]{DMNO}
Alexei Davydov, Michael M\"{u}ger, Dmitri Nikshych, Victor Ostrik. The Witt group of non-degenerate braided fusion categories. Journal f\"{u}r die reine und angewandte Mathematik (Grelles Journal), 2013, 677:135-177.
\bibitem[DNO]{DNO}
Alexei Davydov, Dmitri Nikshych, Victor Ostrik. On the structure of the Witt group of braided fusion categories.
Selecta Mathematica, 2012, 19(1):237-269. 
\bibitem[EGNO]{EGNO}
Pavel Etingof, Shlomo Gelaki, Dmitri Nikshych, Victor Ostrik. Tensor categories. American Mathematical Society, Providence, R.I,, 2015.
\bibitem[F]{F}
John Francis. The tangent complex and Hochschild cohomology of $E_n$-rings. Compositio Mathematica, 2013, 149:430-480.
\bibitem[FG]{FG}
John Francis, Dennis Gaitsgory. Chiral koszul duality. Selecta Mathematica, 2012, 18:27-87.
\bibitem[L]{Lu}
Jacob Lurie. On the classification of topological field theories. Current Developments in Mathematics, 2008, (1):129-280.
\bibitem[KZ]{Liang}
Liang Kong, Hao Zheng. The center functor is fully faithful. Advances in Mathematics, 2018, 339:749-779.
\bibitem[LKW]{TL}
Tian Lan, Liang Kong, Xiao-Gang Wen. Modular extensions of unitary braided fusion categories and 2+1D topological/SPT orders with symmetries. Communications in Mathematical Physics, 2016, 351(2):709-739.
\bibitem[S]{SL}
Long Sun. The symmetric enriched center functor is fully faithful. Communications in Mathematical Physics, 2022, 395:1345-1382.
\bibitem[W]{Wei}
Xiao-Xue Wei. Algebras over a symmetric fusion category and integrations. arXiv:2206.00475.
\end{thebibliography}
\end{document}